\documentclass[10pt, a4paper, abstract=yes]{scrartcl}

\usepackage[utf8]{inputenc}
\usepackage{hyperref}
\usepackage{amsmath}
\usepackage{amsthm}
\usepackage{amssymb}
\usepackage{mathbbol}
\usepackage[inkscapelatex=false]{svg}
\usepackage{graphicx}
\usepackage{float}
\usepackage{tikz}
\usetikzlibrary{arrows}
\usepackage{textcomp}
\usepackage{wrapfig}
\usepackage{xcolor}
\usepackage{ytableau}
\usepackage{tabularx}
\usepackage{geometry}
\usepackage{changepage}
\newcommand{\set}[1]{\left\{{#1}\right\}}
\newcommand\setsuchas[2]{\left\{\,{#1}\,\vrule\,{#2}\,\right\}}

\newcommand{\CC}{\mathbb{C}}

\newcommand{\NN}{\mathbb{N}}

\newcommand{\lenpart}[1]{\mathrm{length}(#1)}
\newcommand{\linext}[1]{\mathcal{L}_{#1}}

\newcommand{\Prob}{\mathbb{P}}
\newcommand{\pof}{V}
\newcommand{\poi}{\Lambda}
\newcommand{\oiof}{\subseteq_I}
\newcommand{\blocking}[3]{\mathcal{BI}_{#1}(#2 < #3)}
\newcommand{\abchain}[3]{\mathcal{CH}_{#1}(#2 < #3)}
\newcommand{\lar}{\tilde{\lambda}}
\newcommand{\lamu}{\lambda / \mu}
\newcommand{\lamur}{\tilde{\lambda} / \tilde{\mu} }
\newcommand{\fla}{f^{\lambda}}
\newcommand{\flamu}{f^{\lambda / \mu}}

\newcommand{\flamur}{f^{\tilde{\lambda} / \tilde{\mu}}}
\newcommand{\Young}{\mathcal{Y}}
\newcommand{\antichain}[2]{#1 \parallel #2}
\theoremstyle{plain}
\newtheorem{theorem}{Theorem}[section]
\newtheorem{lemma}[theorem]{Lemma}
\newtheorem{corollary}[theorem]{Corollary}
\newtheorem{conjecture}[theorem]{Conjecture}

\theoremstyle{definition}
\newtheorem{definition}[theorem]{Definition}
\theoremstyle{remark}
\newtheorem{remark}[theorem]{Remark}
\newtheorem{mexample}[theorem]{Example}

\ytableausetup{smalltableaux}
\author{Albin Jaldevik\thanks{albin.jaldevik@gmail.com}
\and
Jan Snellman\thanks{Department of Mathematics, Linköping University, 58183 Linköping, Sweden;
jan.snellman@liu.se. Corresponding author.}}
\date{\textit{<2026-05-06 Wed>}}
\title{Poset probability in two-row partition posets}
\hypersetup{
 pdfauthor={},
 pdftitle={Poset probability in two-row partition posets},
 pdfkeywords={},
 pdfsubject={},
 pdfcreator={},
 pdflang={English}}
\usepackage[style=numeric]{biblatex}
\begin{document}

\maketitle

\begin{abstract}
We find explicit formulae for poset probabilities \(\Prob(P_\lambda; \alpha < \beta)\)
in partition posets (cell posets) \(P_{\lambda}\) when \(\lambda=(\lambda_{1},\lambda_{2})\) is a two-row partition.
These probabilities are given as rational expressions in \(f^{\sigma / \tau}\), where
\(\tau \subseteq \sigma \subseteq \lambda\).
We then use well-known formulae, such as the hook-length formula for \(\fla\),
the number of standard Young tableaux on a partition \(\lambda\),
and the corresponding determinantal formula by Jacobi-Trudi-Aitken for \(\flamu\),
the number of standard Young tableaux on a skew partition \(\lamu\),
to make the aforementioned expressions explicit.

We also calculate the limit probabilities of \(\Prob(P_\lambda; \alpha < \beta)\)
when the elements \(\alpha,\beta\) are fixed cells,
but the arm-lengths of \(\lambda=(\lambda_{1},\lambda_{2})\) tend to infinity with bounded difference \(\lambda_{1} - \lambda_{2}\).
\end{abstract}
\section{Introduction}
\label{sec:orgecca277}

\subsection{Poset elementa}
\label{sec:org6e72b30}
We will use \(\subset\) for strict inclusion, and \(\subseteq\) for non-strict inclusion.
For a poset \(P = (P, \le)\) we similarly use \(<\) for the \textbf{strict part}, i.e.,
\(x < y\) iff \(x \le y\) and \(x \neq y\). The \textbf{reverse relation} is \(y \ge x \iff x \le y\).

A subset \(U \subseteq P\) is an \textbf{order ideal} if \(u \in U,  v \in P,  v \le u\)
implies that \(v \in U\).
We write \(U \oiof P\)  to indicate that \(U\) is an order ideal of \(P\).
For \(x \in P\), the subset \(\poi_{P}(x) = \{y \in P \mid y \le x\}\)
is the principal (closed) order ideal of \(x\).
Dually, a subset \(V \subseteq P\) is an \textbf{order filter} if \(v \in V, w \in P, v \le w\)
implies that \(w \in V\), and \(\pof_{P}(x) = \{y \in P \mid y \ge x\}\)
is the principal (closed) order filter of \(x\).

The poset \(P\) is a \textbf{lattice} if every two-element subset \(\set{a,b} \subset P\)
has a least upper bound \(a \vee b\) and an greatest lower bound \(a \wedge b\). A lattice
is \textbf{distributive} if the identities
\begin{equation}
\label{eq-distributive}
  \begin{split}
    a \wedge (b \vee c) & = (a \wedge b) \vee (a \wedge c) \\
    a \vee (b \wedge c) & = (a \vee b) \wedge (a \vee c) 
  \end{split}
\end{equation}
hold for all \(a,b,c \in P\).

A subset \(C \subseteq P\) is a \textbf{chain} if \(c,d \in C\) implies that
\(c \le d\) or \(d \le c\). A chain is \textbf{saturated} (or \textbf{maximal}) if
adding any element from \(P \setminus C\) makes it no longer a chain.
If the subset \(P \subseteq P\) is a chain, then we say that the poset \(P\) is a chain.

If \((Q,\le_{Q})\) is another poset, then the \textbf{product} \(P \times Q\) has the Cartesian product
as its underlying set, and order relation
\[
(a,b) \preceq (c,d) \quad  \iff \quad a \le  c \text{ and } b \le_{Q} d.
\]
A product of chains is called a \textbf{chain product}.

The poset \((\NN, \le)\), with \(\le\) the standard ordering, is a chain, denoted by \(C_\infty\).
For a positive integer \(n\), the subset \([n] = \set{1,2,\dots,n} \subset \NN\) is a chain (regarded as a subposet
of \(C_\infty\)), which we denote by \(C_{n}\).

Now suppose that \(P\) is finite, with \(|P| = n\).
By \(\ell(P)\) we mean the number of saturated chains in \(P\).
Denote by \(\linext{P}\) the set of \textbf{linear extensions} (or \textbf{total extensions}) of \(P\), that is, the set
of all  bijections
\begin{equation}
\phi: C_{n} \to P
\end{equation}
with order-preserving inverse;
this corresponds to a sequence
\begin{equation}
\label{eqn-linext-chain}
   \phi(1) , \phi(2) , \cdots , \phi(n), \qquad
i < j \implies  (\phi(i) \not > \phi(j)).
\end{equation}
Denote by \(e(P) = \lvert \linext{P} \rvert\) the cardinality of this set, and,
for \(\alpha,\beta \in P\),
denote by \(\linext{P}(\alpha < \beta)\) the subset consisting
of those linear extensions that place \(\alpha\) before \(\beta\);
such  an extension may be represented as
\begin{equation}
\label{eqn-linext-chain-ab}
   \phi(1) ,\, \phi(2) ,\, \cdots ,\, \phi(k) = \alpha   ,\, \phi(k+1) , \,  \cdots , \, \phi(\ell) = \beta ,\, \phi(\ell + 1) ,\, \cdots , \, \phi(n).
\end{equation}
Denote by \(e(P;\alpha < \beta) = \lvert \linext{P}(\alpha < \beta) \rvert\) the number of such linear extensions.
\subsection{Poset probability}
\label{sec:org0cfbec8}
We define the \textbf{probability}  of \(\alpha\) preceding \(\beta\) as
\begin{equation}
  \Prob(P; \alpha < \beta) =
  \frac{e(P;\alpha < \beta)}{e(P)}.
\end{equation}
If \(\alpha < \beta\) then \(\Prob(P; \alpha < \beta) = 1\), and if \(\alpha \ge \beta\) then \(\Prob(P; \alpha < \beta) = 0\).
The nontrivial case is when \(\alpha,\beta \in P\) are incomparable (written \(x \parallel y\));
then there is at least one linear extension of \(P\) placing \(\alpha\) before \(\beta\),
and at least one placing \(\alpha\) after \(\beta\). In other words,
\(0 < \Prob(P; \alpha < \beta) < 1\) if and only if \(\alpha \parallel \beta\).

We want to find \(\linext{P}(\alpha < \beta)\) and  its cardinality.
One approach, used e.g. by Sagan and Olson \autocite{olson20181},  is to
make use of the following fact:

\begin{lemma}
Denote by \(P(\alpha < \beta)\) the poset obtained by adding \((\alpha,\beta)\) as a
relation to (the graph of) \(P\), as well as all transitive consequences.
Then \(\linext{P}(\alpha < \beta)\)  and \(\linext{P(\alpha < \beta)}\) are in
bijective correspondence, hence \(e(P; \alpha < \beta) = e(P(\alpha < \beta))\).
\end{lemma}
If the poset \(P(\alpha < \beta)\) is easy to construct and analyse, then this is a powerful method. For instance, in Figure \ref{fig-addrel},
those linear extensions of the poset to the left,
that place \((2,1)\) before \((1,2)\),
corresponds to all linear extensions of the poset to the right.

\begin{figure}[htb]
\tiny{
\begin{center}
\begin{tabular}{cc}
\begin{tikzpicture}[>=latex,line join=bevel,]
\node (node_0) at (60.39bp,8.3018bp) [draw,draw=none] {$\left(1, 1\right)$};
  \node (node_1) at (37.39bp,60.906bp) [draw,draw=none] {$\left(1, 2\right)$};
  \node (node_4) at (84.39bp,60.906bp) [draw,draw=none] {$\left(2, 1\right)$};
  \node (node_2) at (14.39bp,113.51bp) [draw,draw=none] {$\left(1, 3\right)$};
  \node (node_5) at (61.39bp,113.51bp) [draw,draw=none] {$\left(2, 2\right)$};
  \node (node_3) at (14.39bp,166.11bp) [draw,draw=none] {$\left(1, 4\right)$};
  \node (node_6) at (108.39bp,113.51bp) [draw,draw=none] {$\left(3, 1\right)$};
  \draw [black,->] (node_0) ..controls (54.102bp,23.136bp) and (49.301bp,33.7bp)  .. (node_1);
  \draw [black,->] (node_0) ..controls (66.951bp,23.136bp) and (71.962bp,33.7bp)  .. (node_4);
  \draw [black,->] (node_1) ..controls (31.102bp,75.74bp) and (26.301bp,86.304bp)  .. (node_2);
  \draw [black,->] (node_1) ..controls (43.951bp,75.74bp) and (48.962bp,86.304bp)  .. (node_5);
  \draw [black,->] (node_2) ..controls (14.39bp,128.04bp) and (14.39bp,137.98bp)  .. (node_3);
  \draw [black,->] (node_4) ..controls (78.102bp,75.74bp) and (73.301bp,86.304bp)  .. (node_5);
  \draw [black,->] (node_4) ..controls (90.951bp,75.74bp) and (95.962bp,86.304bp)  .. (node_6);
\end{tikzpicture}
 & 
\begin{tikzpicture}[>=latex,line join=bevel,]
\node (node_0) at (60.39bp,8.3018bp) [draw,draw=none] {$\left(1, 1\right)$};
  \node (node_1) at (60.39bp,60.906bp) [draw,draw=none] {$\left(2, 1\right)$};
  \node (node_2) at (37.39bp,113.51bp) [draw,draw=none] {$\left(1, 2\right)$};
  \node (node_6) at (84.39bp,113.51bp) [draw,draw=none] {$\left(3, 1\right)$};
  \node (node_3) at (14.39bp,166.11bp) [draw,draw=none] {$\left(1, 3\right)$};
  \node (node_5) at (61.39bp,166.11bp) [draw,draw=none] {$\left(2, 2\right)$};
  \node (node_4) at (14.39bp,218.72bp) [draw,draw=none] {$\left(1, 4\right)$};
  \draw [black,->] (node_0) ..controls (60.39bp,22.834bp) and (60.39bp,32.769bp)  .. (node_1);
  \draw [black,->] (node_1) ..controls (54.102bp,75.74bp) and (49.301bp,86.304bp)  .. (node_2);
  \draw [black,->] (node_1) ..controls (66.951bp,75.74bp) and (71.962bp,86.304bp)  .. (node_6);
  \draw [black,->] (node_2) ..controls (31.102bp,128.34bp) and (26.301bp,138.91bp)  .. (node_3);
  \draw [black,->] (node_2) ..controls (43.951bp,128.34bp) and (48.962bp,138.91bp)  .. (node_5);
  \draw [black,->] (node_3) ..controls (14.39bp,180.64bp) and (14.39bp,190.58bp)  .. (node_4);
\end{tikzpicture}
\end{tabular}
\end{center}
}
\caption{Adding a relation to a poset forces linear extensions}
\label{fig-addrel}
\end{figure}

Another way of calculating \(e(P; \alpha < \beta)\) is implicit in \autocite{kangasCountingLinearExtensions}; similar ideas
are used in \autocite{deloofExploitingLatticeIdeals2006} and in \autocite{peczarskiNewResultsMinimumComparison2004a}.

\begin{remark}
When writing the first version of this manuscript,
we were unaware of these references, and re-discovered this method,
generalising from more specialised lattice-path techniques.
\end{remark}

Let \(D,U\) be a set partition of \(P \setminus \set{\alpha}\),
such that \(D\) is an order ideal containing \(\poi_{P}(\alpha) \setminus \set{\alpha}\),
\(U\) is an order filter containing \(\pof_{P}(\alpha) \setminus \set{\alpha}\) and \(\beta\).
Then a linear extension of \(D\) and a linear extension of \(U\)
can be spliced together to produce a linear extension of \(P\) placing \(\alpha\) before \(\beta\). Furthermore, all linear extension of \(P\) placing \(\alpha\) before \(\beta\)
are obtained in this way, so
\begin{equation}\tag{AP}
\label{eq:admissible-partition}
e(P; \alpha < \beta) = \sum_{{(D,U)}} e(D)\cdot e(U),
\end{equation}
the sum taken over all admissible partitions. This method is especially useful when \(e(D)\) and \(e(U)\) are easily calculable. We will apply it to \textbf{partition posets}, for which there are a plethora of useful formulas for \(e(D)\) and \(e(U)\).
\subsection{Partition posets}
\label{sec:orgbefd2b9}
A partition \(\lambda\) is a finite, weakly decreasing sequence of
positive integers (called \textbf{parts})
\begin{equation}
\lambda=(\lambda_{1},\lambda_{2},\dots,\lambda_{d}), \qquad \lambda_{1} \ge \lambda_{2} \ge \cdots \ge \lambda_{d} >0.
\end{equation}
The integer \(d=\lenpart{\lambda}\) is the \textbf{length}  of \(\lambda\), and
the \textbf{weight} of \(\lambda\) is \(|\lambda| = \lambda_{1} + \lambda_{2} + \cdots + \lambda_{d}\).
We write \(\lambda \vdash n\) and say that \(\lambda\) is a partition of \(n\).

The (Ferrers, or Young) \textbf{diagram} of \(\lambda\) is the set of coordinates, or boxes, or cells,
\begin{equation}
  \{(i,j) \mid 1 \le i \le d, \, 1 \le j \le \lambda_{i}\} \subset [d] \times [\lambda_{1}].
\end{equation}  
In \textbf{British notation} the diagram of \(\lambda\) is displayed as follows:
the cell \((i,j)\) sits in row \(i\), column \(j\), with the rows extending downwards.
By flipping the Ferrers diagram along its main diagonal,
one obtains the \textbf{conjugate partition}
\begin{equation}
    \label{eq:conjugate}
    \lambda'= (\lambda_{1}', \dots, \lambda_{s}'),
    \quad \lambda_{j}' = \max(\{i \mid \lambda_{i} \ge j\}). 
\end{equation}

Conjugation is an involution on the set of all partitions of \(n\).
As an example,
\[
\lambda = (4,2,1) =  \ydiagram{4,2,1}, \qquad \lambda' = (3,2,1,1) = \ydiagram{3,2,1,1}.
\]

The cells of \(\lambda\) form a subset of the elements of \(C_{d} \times C_{\lambda_{1}}\),
the Cartesian product of two chains. The induced sub-poset is denoted \(P_{\lambda}\) and
is called the \textbf{partition poset}, or \textbf{cell poset}, of \(\lambda\). The Hasse diagram of
\(P_{{(4,2,1)}}\) is shown in Figure \ref{fig:p421}.

\begin{figure}[htb]
\tiny{
\begin{center}
\begin{tikzpicture}[>=latex,line join=bevel,]
\node (node_0) at (60.39bp,8.3018bp) [draw,draw=none] {$\left(1, 1\right)$};
  \node (node_1) at (37.39bp,60.906bp) [draw,draw=none] {$\left(1, 2\right)$};
  \node (node_4) at (84.39bp,60.906bp) [draw,draw=none] {$\left(2, 1\right)$};
  \node (node_2) at (14.39bp,113.51bp) [draw,draw=none] {$\left(1, 3\right)$};
  \node (node_5) at (61.39bp,113.51bp) [draw,draw=none] {$\left(2, 2\right)$};
  \node (node_3) at (14.39bp,166.11bp) [draw,draw=none] {$\left(1, 4\right)$};
  \node (node_6) at (108.39bp,113.51bp) [draw,draw=none] {$\left(3, 1\right)$};
  \draw [black,->] (node_0) ..controls (54.102bp,23.136bp) and (49.301bp,33.7bp)  .. (node_1);
  \draw [black,->] (node_0) ..controls (66.951bp,23.136bp) and (71.962bp,33.7bp)  .. (node_4);
  \draw [black,->] (node_1) ..controls (31.102bp,75.74bp) and (26.301bp,86.304bp)  .. (node_2);
  \draw [black,->] (node_1) ..controls (43.951bp,75.74bp) and (48.962bp,86.304bp)  .. (node_5);
  \draw [black,->] (node_2) ..controls (14.39bp,128.04bp) and (14.39bp,137.98bp)  .. (node_3);
  \draw [black,->] (node_4) ..controls (78.102bp,75.74bp) and (73.301bp,86.304bp)  .. (node_5);
  \draw [black,->] (node_4) ..controls (90.951bp,75.74bp) and (95.962bp,86.304bp)  .. (node_6);
\end{tikzpicture}
\end{center}
}
\caption{The Hasse diagram of the partition poset on \(\lambda=(4,2,1)\)}
\label{fig:p421}
\end{figure}

When \(\mu, \lambda\) are partitions, we say that \(\mu\) is a \textbf{subpartition} of \(\lambda\),
written \(\mu \subseteq \lambda\), if \(\mu_{i} \le \lambda_{i}\)  for all \(i\); this occurs precisely when
\(P_\mu \oiof P_\lambda\). The \textbf{skewpartition} \(\lambda / \mu\) is the ordered pair \((\lambda,\mu)\);
its partition poset \(P_{\lambda/\mu}\) is the subposet of \(P_\lambda\) consisting of those cells of
\(P_\lambda\) that do not belong to \(P_\mu\).
\subsection{Linear extensions of partition posets}
\label{sec:org26eda50}
A linear extension \(\sigma: C_{n} \to P_\lambda\) of \(P_\lambda\) correspond to a
a saturated chain of subpartitions of \(\lambda\). This can be depicted as a
\textbf{Standard Young Tableaux} (SYT) on \(\lambda\), where the cell \((i,j)\) is labeled with \(\sigma^{-1}(i,j)\).

Furthermore, if \(\alpha\) and \(\beta\) are two incomparable elements, then
\(\linext{P}(\alpha < \beta)\) corresponds to SYT on \(\lambda\) with smaller value
at position \(\alpha\) than at position \(\beta\).

The so-called \emph{Young's lattice} \(\Young\),
the first ranks of which are shown in Figure \ref{fig-Younglattice},
consists of all partitions,
partially ordered by inclusion.
Thus a linear extension of \(P_{\lambda}\) correspond (in addition to a SYT on \(\lambda\))
to a saturated chain in \(\Young\) starting
from the unique minimal element \(\emptyset\) and ending at \(\lambda\).

\begin{figure}[htb]
\begin{center}
\resizebox{!}{5cm}{%
\begin{tikzpicture}[>=latex,line join=bevel,]
\node (node_0) at (95.841bp,7.3332bp) [draw,draw=none] {${\emptyset}$};
  \node (node_1) at (95.841bp,59.965bp) [draw,draw=none] {${\def\lr#1{\multicolumn{1}{|@{\hspace{.6ex}}c@{\hspace{.6ex}}|}{\raisebox{-.3ex}{$#1$}}}\raisebox{-.6ex}{$\begin{array}[b]{*{1}c}\cline{1-1}\lr{\phantom{x}}\\\cline{1-1}\end{array}$}}$};
  \node (node_2) at (75.841bp,120.54bp) [draw,draw=none] {${\def\lr#1{\multicolumn{1}{|@{\hspace{.6ex}}c@{\hspace{.6ex}}|}{\raisebox{-.3ex}{$#1$}}}\raisebox{-.6ex}{$\begin{array}[b]{*{1}c}\cline{1-1}\lr{\phantom{x}}\\\cline{1-1}\lr{\phantom{x}}\\\cline{1-1}\end{array}$}}$};
  \node (node_5) at (116.84bp,120.54bp) [draw,draw=none] {${\def\lr#1{\multicolumn{1}{|@{\hspace{.6ex}}c@{\hspace{.6ex}}|}{\raisebox{-.3ex}{$#1$}}}\raisebox{-.6ex}{$\begin{array}[b]{*{2}c}\cline{1-2}\lr{\phantom{x}}&\lr{\phantom{x}}\\\cline{1-2}\end{array}$}}$};
  \node (node_3) at (49.841bp,193.07bp) [draw,draw=none] {${\def\lr#1{\multicolumn{1}{|@{\hspace{.6ex}}c@{\hspace{.6ex}}|}{\raisebox{-.3ex}{$#1$}}}\raisebox{-.6ex}{$\begin{array}[b]{*{1}c}\cline{1-1}\lr{\phantom{x}}\\\cline{1-1}\lr{\phantom{x}}\\\cline{1-1}\lr{\phantom{x}}\\\cline{1-1}\end{array}$}}$};
  \node (node_6) at (96.841bp,193.07bp) [draw,draw=none] {${\def\lr#1{\multicolumn{1}{|@{\hspace{.6ex}}c@{\hspace{.6ex}}|}{\raisebox{-.3ex}{$#1$}}}\raisebox{-.6ex}{$\begin{array}[b]{*{2}c}\cline{1-2}\lr{\phantom{x}}&\lr{\phantom{x}}\\\cline{1-2}\lr{\phantom{x}}\\\cline{1-1}\end{array}$}}$};
  \node (node_4) at (8.8408bp,277.55bp) [draw,draw=none] {${\def\lr#1{\multicolumn{1}{|@{\hspace{.6ex}}c@{\hspace{.6ex}}|}{\raisebox{-.3ex}{$#1$}}}\raisebox{-.6ex}{$\begin{array}[b]{*{1}c}\cline{1-1}\lr{\phantom{x}}\\\cline{1-1}\lr{\phantom{x}}\\\cline{1-1}\lr{\phantom{x}}\\\cline{1-1}\lr{\phantom{x}}\\\cline{1-1}\end{array}$}}$};
  \node (node_7) at (49.841bp,277.55bp) [draw,draw=none] {${\def\lr#1{\multicolumn{1}{|@{\hspace{.6ex}}c@{\hspace{.6ex}}|}{\raisebox{-.3ex}{$#1$}}}\raisebox{-.6ex}{$\begin{array}[b]{*{2}c}\cline{1-2}\lr{\phantom{x}}&\lr{\phantom{x}}\\\cline{1-2}\lr{\phantom{x}}\\\cline{1-1}\lr{\phantom{x}}\\\cline{1-1}\end{array}$}}$};
  \node (node_9) at (148.84bp,193.07bp) [draw,draw=none] {${\def\lr#1{\multicolumn{1}{|@{\hspace{.6ex}}c@{\hspace{.6ex}}|}{\raisebox{-.3ex}{$#1$}}}\raisebox{-.6ex}{$\begin{array}[b]{*{3}c}\cline{1-3}\lr{\phantom{x}}&\lr{\phantom{x}}&\lr{\phantom{x}}\\\cline{1-3}\end{array}$}}$};
  \node (node_8) at (96.841bp,277.55bp) [draw,draw=none] {${\def\lr#1{\multicolumn{1}{|@{\hspace{.6ex}}c@{\hspace{.6ex}}|}{\raisebox{-.3ex}{$#1$}}}\raisebox{-.6ex}{$\begin{array}[b]{*{2}c}\cline{1-2}\lr{\phantom{x}}&\lr{\phantom{x}}\\\cline{1-2}\lr{\phantom{x}}&\lr{\phantom{x}}\\\cline{1-2}\end{array}$}}$};
  \node (node_10) at (148.84bp,277.55bp) [draw,draw=none] {${\def\lr#1{\multicolumn{1}{|@{\hspace{.6ex}}c@{\hspace{.6ex}}|}{\raisebox{-.3ex}{$#1$}}}\raisebox{-.6ex}{$\begin{array}[b]{*{3}c}\cline{1-3}\lr{\phantom{x}}&\lr{\phantom{x}}&\lr{\phantom{x}}\\\cline{1-3}\lr{\phantom{x}}\\\cline{1-1}\end{array}$}}$};
  \node (node_11) at (211.84bp,277.55bp) [draw,draw=none] {${\def\lr#1{\multicolumn{1}{|@{\hspace{.6ex}}c@{\hspace{.6ex}}|}{\raisebox{-.3ex}{$#1$}}}\raisebox{-.6ex}{$\begin{array}[b]{*{4}c}\cline{1-4}\lr{\phantom{x}}&\lr{\phantom{x}}&\lr{\phantom{x}}&\lr{\phantom{x}}\\\cline{1-4}\end{array}$}}$};
  \draw [black,->] (node_0) ..controls (95.841bp,20.716bp) and (95.841bp,30.515bp)  .. (node_1);
  \draw [black,->] (node_1) ..controls (90.771bp,75.813bp) and (87.401bp,85.683bp)  .. (node_2);
  \draw [black,->] (node_1) ..controls (101.68bp,77.24bp) and (106.26bp,90.026bp)  .. (node_5);
  \draw [black,->] (node_2) ..controls (67.858bp,143.19bp) and (64.424bp,152.51bp)  .. (node_3);
  \draw [black,->] (node_2) ..controls (82.738bp,144.7bp) and (86.218bp,156.39bp)  .. (node_6);
  \draw [black,->] (node_3) ..controls (35.801bp,222.31bp) and (28.831bp,236.33bp)  .. (node_4);
  \draw [black,->] (node_3) ..controls (49.841bp,223.39bp) and (49.841bp,234.75bp)  .. (node_7);
  \draw [black,->] (node_5) ..controls (111.96bp,138.75bp) and (107.63bp,154.03bp)  .. (node_6);
  \draw [black,->] (node_5) ..controls (125.38bp,140.36bp) and (134.12bp,159.62bp)  .. (node_9);
  \draw [black,->] (node_6) ..controls (82.816bp,218.68bp) and (74.276bp,233.67bp)  .. (node_7);
  \draw [black,->] (node_6) ..controls (96.841bp,219.79bp) and (96.841bp,236.85bp)  .. (node_8);
  \draw [black,->] (node_6) ..controls (113.43bp,220.38bp) and (124.9bp,238.58bp)  .. (node_10);
  \draw [black,->] (node_9) ..controls (148.84bp,213.56bp) and (148.84bp,234.35bp)  .. (node_10);
  \draw [black,->] (node_9) ..controls (165.21bp,215.49bp) and (185.28bp,241.77bp)  .. (node_11);
\end{tikzpicture}
}
\end{center}
\caption{The first ranks of the Young lattice}
\label{fig-Younglattice}
\end{figure}

\begin{mexample}
The saturated chain (of partitions)
\begin{equation*}
  \emptyset \subset \ydiagram{1} \subset \ydiagram{2} \subset \ydiagram{3} \subset \ydiagram{3,1} 
\end{equation*}
corresponds to
\begin{equation*}
\begin{ytableau}
  1 & 2 & 3 \\
  4
\end{ytableau},
\end{equation*}
whereas
\begin{equation*}
  \emptyset \subset \ydiagram{1} \subset \ydiagram{1,1} \subset \ydiagram{2,1}  \subset \ydiagram{3,1}
\end{equation*}
corresponds to
\begin{equation*}
\begin{ytableau}
  1 & 3 & 4 \\
  2
\end{ytableau}.
\end{equation*}
\label{example-SYT}
\end{mexample}

It is also of interest to study saturated chains in other finite intervals of \(\Young\).
Such chains can be represented as \emph{standard skew tableaux}.

\begin{mexample}
The saturated chains
\begin{equation*}
\ydiagram{2}  \subset \ydiagram{3} \subset \ydiagram{3,1}  \subset \ydiagram{3,2} \quad \text{ and } \quad
\ydiagram{2}  \subset \ydiagram{2,1} \subset \ydiagram{2,2} \subset \ydiagram{3,2}
\end{equation*}
correspond to
\begin{equation*}
  \begin{ytableau}
    \none & \none &1 \\
    2 & 3
  \end{ytableau}
  \quad \text{ and } \quad
  \begin{ytableau}
    \none & \none & 3 \\
    1 & 2
  \end{ytableau}.
\end{equation*}
\end{mexample}
\subsection{Poset probability for partition posets}
\label{sec:org6e8e1e4}
For a partition poset \(P_{\lambda}\) there are several well-known formulae for counting
\(e(P_{\lambda})\). These methods extend to skew-partitions \(\lambda / \mu\).
To count \(e(P_{\lambda}; \alpha < \beta)\), we use a version of (\ref{eq:admissible-partition}),
which now becomes
\begin{equation}
e(P_{\lambda}; \alpha < \beta) =  \sum_{T \in \mathcal{B}} f^{T} f^{\lambda / (T \cup \set{\alpha})}.
\end{equation}
The set \(\mathcal{B}\) of \textbf{blocking partitions} can be found using the \textbf{decorated tableau}
of \(\lambda\), see Example \ref{example-lam-small}. 

In this article, we will calculate exact values for \(\Prob({P_{\lambda}}; \alpha < \beta)\) when \(\lambda=(\lambda_{1},\lambda_{2})\).
We will use this to compute the limit for \(\Prob({P_{\lambda}}; \alpha < \beta)\)
when \(\lambda = (\lambda_{1},\lambda_{2})\) tends to \((\infty,\infty)\) in such a way that \(\lambda_{1} - \lambda_{2}\) remains bounded.
The fact that this limit exists shows that it can be calculated as the limit of
\(\Prob({C_{m} \times C_{2}}; \alpha < \beta)\) as \(m \to \infty\);
the poset \(C_{m} \times C_{2}\) is often called the ``Catalan poset''.
Sagan and Olson \autocite{olson20181}  proved an exact formula for \(\Prob(P_{\lambda}; \alpha < \beta)\)
when \(\lambda\) is of \textbf{rectangular shape}, i.e. \(P_{\lambda} = C_{m} \times C_{n}\), and when \(\alpha=(1,2)\),
\(\beta=(2,1)\).

We note that there is a decently sized body of literature on asymptotic results for  probabilities on partition posets \autocite{Chan-catalan,chanSortingProbabilityLarge}.
\section{Blocking ideals - filtering linear extensions}
\label{sec:orgc36e96e}
We now describe in more detail how equation (\ref{eq:admissible-partition}) can be used to
find \(\linext{P}(\alpha < \beta)\).  The easy proofs
of results in this section are omitted.
\subsection{Blocking ideals and linear extensions}
\label{sec:org3c60239}
Recall the following results from lattice theory.
\begin{definition}
Let \(P\) be a finite poset, and put
\(J(P) = \{U \subseteq P \, \mid \, U \oiof P\}\).
Order these order ideals by inclusion to turn \(J(P)\)
into a poset.
\label{def-JP}
\end{definition}

\begin{theorem}[Birkhoff's representation theorem]
The poset \(J(P)\) is a distributive lattice. Conversely, any
finite distributive lattice \(L\) is of the form \(L=J(P)\)
for some finite poset \(P\).
\label{birk-repr}
\end{theorem}
\begin{proof}
See  \autocite{BirkhoffRepr}.
\end{proof}
\begin{theorem}
Let \(P\) be a poset with \(n < \infty\) elements.
Linear extensions in \(\linext{P}\) correspond to saturated chains in \(J(P)\),
by pairing the linear extension \(\phi: C_{n} \to P\) with the saturated chain
\begin{equation}
\emptyset \subset \{\phi(1)\} \subset \{\phi(1),\phi(2)\} \cdots \subset \{\phi(1),\dots,\phi(n)\} = P
\end{equation}
In particular, \(e(P) = \ell(J(P))\).
\end{theorem}
\begin{proof}
See  \autocite{stanley-enum1}.
\end{proof}

Now let
\(\sigma \in \linext{P} (\alpha < \beta)\), and represent it by a sequentially ordered list of elements in \(P\),
as in (\ref{eqn-linext-chain-ab}).
We want to split this linear extension into two parts; from the beginning up
to \(\alpha\), and then from \(\alpha\) to the end. If we pass to \(J(P)\),
then the linear extension \(\sigma\) correspond to a \emph{saturated chain}  of  order ideals in \(J(P)\)
\begin{equation}
\label{eq:JP-chain}
\emptyset \subset U_{1} \subset \cdots U_{i-1} \subset U_{i} \subset \cdots \subset U_{n} = P
\end{equation}
where \(\alpha, \beta \not \in U_{{i-1}}\), \(\alpha \in U_{i}\). That is
to say, \(\alpha\) is introduced at stage \(i\), and \(\beta\) at a later stage.

The order ideal \(U_{i-1}\) is covered by \(U_{i}\), and is the last
order ideal in the chain that does not contain \(\alpha\). We call it a
\emph{blocking ideal}.

\begin{definition}
Let \(\alpha,\beta\) be two incomparable elements in the finite poset \(P\).
The \textbf{blocking ideals} \(\blocking{P}{\alpha}{\beta}\) is the set of order ideals of \(P\) without \(\alpha\) and \(\beta\),
and where we can add \(\alpha\) and get a new order ideal:
\begin{equation}
\label{eq-block-def}
\blocking{P}{\alpha}{\beta} \triangleq
\setsuchas{T \in J(P)}{\alpha \notin T,\, \beta \notin T,\, T \cup \set{\alpha} \in J(P)}.
\end{equation}
\label{def-blocking-ideals}
\end{definition}

Comparing this with equation (\ref{eq:admissible-partition}),
the blocking ideal \(T\) corresponds to \(D\),
and \(U\) to  \(P \setminus (D \cup \set{\alpha})\).
Hence, these blocking ideals enable us to describe \(\linext{P}(\alpha < b)\).

\begin{definition}
When \(P\) is a finite poset and \(\alpha,\beta\) are two incomparable elements in \(P\),
the set of saturated chains in \(J(P)\) which add \(\alpha\) before \(\beta\)
is denoted \(\abchain{P}{\alpha}{\beta}\).
By \(C_P(\alpha < \beta)\) we denote the set of chains in \(J(P)\) defined by
\begin{equation}
\label{eq:chain-refine}
C_{P}(\alpha < \beta) \triangleq
\setsuchas{\emptyset \subset T \subset T \cup \set{\alpha} \subset P}
          {T \in \blocking{P}{\alpha}{\beta}}.
\end{equation}
\label{def-abchains}
\end{definition}

\begin{theorem}
Let \(P\) be a finite poset and let \(\alpha, \beta \in P\) be two incomparable elements.
Then the following hold:
\begin{enumerate}
\item Every element in \(\abchain{P}{\alpha}{\beta}\) (a saturated chain in \(J(P)\) adding \(\alpha\) before \(\beta\)) is a refinement of precisely one chain in \(C_P(\alpha < \beta)\).

\item The total number of saturated chains in \(J(P)\) adding \(\alpha\) before \(\beta\) can be calculated as
\end{enumerate}

\begin{equation}
\label{eq:block-count}
|\abchain{P}{\alpha}{\beta}| =
\sum_{T \in \blocking{P}{\alpha}{\beta}}
     \ell(\poi_{J(P)}(T)) \ell(\pof_{J(P)}(T \cup \set{\alpha})).
\end{equation}
\label{thm-blocking-refines}
\end{theorem}

\begin{corollary}
Let \(P\) be a finite poset and let \(\alpha, \beta \in P\)
be two incomparable elements. Then

\begin{equation}
\label{eq:e-refine}
e(P; \alpha < \beta) =
\sum_{T \in \blocking{P}{\alpha}{\beta}} e(T)
e(P \setminus (T \cup \set{\alpha})). \tag{LINab}
\end{equation}
\label{cor-refine}
\end{corollary}
\subsection{The structure of blocking ideals}
\label{sec:org42c145e}
The following description of blocking ideals provides a way of finding all
blocking ideals in a finite poset \(P\) with a fixed antichain \(a \parallel b\).
In the next section, we will apply it to partition posets to obtain a
\textbf{blocking expansion}, allowing us to calculate \(e(P; \alpha < \beta)\).
\begin{definition}
Let \(P\) be a finite poset and let \(\alpha,\beta \in P\) be two incomparable elements.
\begin{enumerate}
\item The \textbf{complete ideal} \(D_P(\alpha, \beta)\) is defined as
\begin{equation} \tag{COM}
\label{eq-com}
     D_P(\alpha,\beta) \triangleq
     P \setminus (\pof_P(\alpha) \cup \pof_P(\beta))
     = \{z \in P \, \mid \, z \not \ge \alpha, \, z \not \ge \beta\}.
\end{equation}
\item The \textbf{constant ideal} \(A_P(\alpha)\), also called the \emph{fixed part}, is the open
principal order ideal of \(\alpha\):
\begin{equation} \tag{FIX}
\label{eq-fix}
     A_P(\alpha) \triangleq \poi_P(\alpha) \setminus \set{\alpha} = \{z \in P \, \mid \, z < \alpha\}.
\end{equation}
\item The \textbf{variable part} \(G_P(\alpha,\beta)\) is
\begin{equation} \tag{VAR}
\label{eq-var}
  G_P(\alpha,\beta) \triangleq D_P(\alpha,\beta) \setminus A_P(\alpha)
  = \{z \in P \, \mid \, z \not \ge \alpha, z \not \ge \beta, z \not \le \alpha\}.
\end{equation}
\end{enumerate}
\label{def-complete-constant-variable}
\end{definition}

\begin{theorem}[Structure of blocking ideals]
Let \(P\) be a finite poset, and let \(a,b \in P\), \(a \parallel b\).
Then \(T \in J(P)\) is a blocking ideal
if and only if it fulfills the following equivalent conditions:

\begin{enumerate}
\item \(A_{P}(\alpha) \subseteq T \oiof D_{P}(\alpha,\beta)\),
(\(T\) is an order ideal in \(D_{P}(\alpha,\beta)\) containing \(A_{p}(\alpha)\))

\item \(T = A_{{P}}(\alpha) \cup V\) with \(V\) an order ideal in \(G_{P}(\alpha,\beta)\).
\end{enumerate}
\label{thm-blocking-structure}
\end{theorem}
\section{Partition posets}
\label{sec:orgd4aa562}
\subsection{Blocking ideals for partition posets.}
\label{sec:org54a2185}
Let \(\lambda\) be a partition and \(P=P_{\lambda}\) its associated cell poset;
suppose that \(\alpha, \beta  \in P_{\lambda}\), with \(\antichain{\alpha}{\beta}\).
Then linear extensions of \(P_{\lambda}\) correspond to SYT on \(\lambda\),
and linear extensions of \(P_{\lambda}\) that place \(\alpha\) before \(\beta\)
correspond to SYT on \(\lambda\) with lower value at cell \(\alpha\) than at cell \(\beta\).
Blocking ideals \(T \in \blocking{P}{\alpha}{\beta}\) are subpartitions  \(T \subset \alpha\).
Subsets of \(P\) of the form \(P \setminus (T \cup \{\alpha\})\) correspond to  skew partitions.

The fixed part \(A_{P}(\alpha)\) is a subpartition of \(\lambda\),
and so is the  complete ideal \(D_{P}(\alpha)\); the variable part \(G_P(\alpha,\beta)\) is not a subpartition.

A blocking ideal \(T\)  contains all of the fixed part, and some of the variable
part, and is a subpartition of the complete ideal.

\begin{figure}[htbp]
\centering
\includegraphics[width=0.25\textwidth]{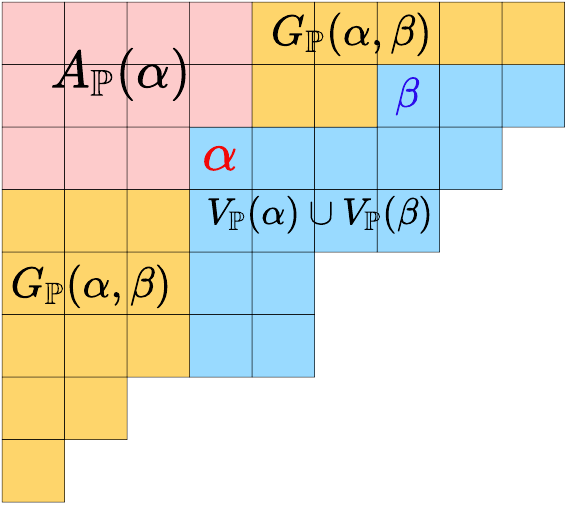}
\caption{\label{fig-part}The complete ideal, fixed part, and variable part of a partition poset \(P=P_\lambda\) with \(\lambda = (9,9,8,7,5,5,2,1)\), \(\alpha = (3,4)\) and \(\beta = (2,7)\). \(P_\lambda\) is partitioned into three disjoint parts. The complete ideal \(D_{P}(\alpha,\beta)\) equals \(A_P(\alpha) \cup G_P(\alpha,\beta)\) which equals \(P_\lambda \setminus (\pof_{P}(\alpha) \cup \pof_{}(\beta))\).}
\end{figure}

\begin{definition}
Let \(\lambda\) be a (number) partition, and \(P_\lambda\) the corresponding cell poset.
We define the \textbf{decorated tableau} of \(\lambda\) as the tableau with shape \(\lambda\),
where the cells are marked as follows:
fixed part:  'F', variable part: 'V', \(\alpha\): '1', \(\beta\): '2', remaining cells: 'o'.
\label{def-decorated}
\end{definition}

\begin{mexample}
Let \(\lambda = (4,3,3)\) and put \(P=P_{\lambda}\).  Let \(\alpha=(2,3)\), \(\beta=(3,2)\).

The decorated tableau is
\phantomsection
\label{}
\ytableausetup{nosmalltableaux}
\begin{center}
\begin{ytableau}
*(yellow) F & *(yellow) F & *(yellow) F & *(gray) V \\ 
*(yellow) F & *(yellow) F & *(red) 1 \\ 
*(gray) V & *(green) 2 & *(cyan) o
\end{ytableau}
\end{center}
\ytableausetup{smalltableaux}

From this representation, we immediately get the blocking ideals, which are
all possible combinations of ``all F + an order ideal of V''. These blocking
ideals are subpartitions of \(\lambda\).

\phantomsection
\label{table-dectab-more}
\begin{center} 
 \begin{ytableau}
X & X & X & X \\ 
X & X & O \\ 
X & O & O
\end{ytableau} 
 \begin{ytableau}
X & X & X & X \\ 
X & X & O \\ 
O & O & O
\end{ytableau} 
 \begin{ytableau}
X & X & X & O \\ 
X & X & O \\ 
O & O & O
\end{ytableau} 
 \begin{ytableau}
X & X & X & O \\ 
X & X & O \\ 
X & O & O
\end{ytableau} 
 \end{center}

The set of SYT's on \(\lambda\) with lower value in \(\alpha\) than in \(\beta\) is the disjoint union of
SYT's which extend a particular blocking partition. For instance, there are precisely
\(5 \times 4 = 20\) SYT extending the third blocking partition \(T\) in the above list, and they
can be ``spliced together'' from the SYT's on \(T\) and on  \(\{\alpha\}\) and on
\(\lambda \setminus (T \cup \{\alpha\})\).
Note that the four skew partitions \(\lambda \setminus (T \cup \{\alpha\}\) are disconnected.
In detail, the 20 SYT are obtained by choosing
first: one of 
\begin{displaymath}
{\def\lr#1{\multicolumn{1}{|@{\hspace{.6ex}}c@{\hspace{.6ex}}|}{\raisebox{-.3ex}{$#1$}}}
\raisebox{-.6ex}{$\begin{array}[b]{*{3}c}\cline{1-3}
\lr{1}&\lr{3}&\lr{5}\\\cline{1-3}
\lr{2}&\lr{4}\\\cline{1-2}
\end{array}$}
}
\quad 
{\def\lr#1{\multicolumn{1}{|@{\hspace{.6ex}}c@{\hspace{.6ex}}|}{\raisebox{-.3ex}{$#1$}}}
\raisebox{-.6ex}{$\begin{array}[b]{*{3}c}\cline{1-3}
\lr{1}&\lr{2}&\lr{5}\\\cline{1-3}
\lr{3}&\lr{4}\\\cline{1-2}
\end{array}$}
}
\quad 
{\def\lr#1{\multicolumn{1}{|@{\hspace{.6ex}}c@{\hspace{.6ex}}|}{\raisebox{-.3ex}{$#1$}}}
\raisebox{-.6ex}{$\begin{array}[b]{*{3}c}\cline{1-3}
\lr{1}&\lr{3}&\lr{4}\\\cline{1-3}
\lr{2}&\lr{5}\\\cline{1-2}
\end{array}$}
}
\quad 
{\def\lr#1{\multicolumn{1}{|@{\hspace{.6ex}}c@{\hspace{.6ex}}|}{\raisebox{-.3ex}{$#1$}}}
\raisebox{-.6ex}{$\begin{array}[b]{*{3}c}\cline{1-3}
\lr{1}&\lr{2}&\lr{4}\\\cline{1-3}
\lr{3}&\lr{5}\\\cline{1-2}
\end{array}$}
}
\quad 
{\def\lr#1{\multicolumn{1}{|@{\hspace{.6ex}}c@{\hspace{.6ex}}|}{\raisebox{-.3ex}{$#1$}}}
\raisebox{-.6ex}{$\begin{array}[b]{*{3}c}\cline{1-3}
\lr{1}&\lr{2}&\lr{3}\\\cline{1-3}
\lr{4}&\lr{5}\\\cline{1-2}
\end{array}$}
}
\quad 
,\end{displaymath}
then the cell corresponding to \(\alpha\):
\begin{displaymath}
{\def\lr#1{\multicolumn{1}{|@{\hspace{.6ex}}c@{\hspace{.6ex}}|}{\raisebox{-.3ex}{$#1$}}}
\raisebox{-.6ex}{$\begin{array}[b]{*{1}c}\cline{1-1}
\lr{6}\\\cline{1-1}
\end{array}$}
}
,\end{displaymath}
finally: one of 
\begin{displaymath}
{\def\lr#1{\multicolumn{1}{|@{\hspace{.6ex}}c@{\hspace{.6ex}}|}{\raisebox{-.3ex}{$#1$}}}
\raisebox{-.6ex}{$\begin{array}[b]{*{4}c}\cline{4-4}
&&&\lr{8}\\\cline{4-4}
&&\\\cline{1-3}
\lr{7}&\lr{9}&\lr{10}\\\cline{1-3}
\end{array}$}
}
\quad 
{\def\lr#1{\multicolumn{1}{|@{\hspace{.6ex}}c@{\hspace{.6ex}}|}{\raisebox{-.3ex}{$#1$}}}
\raisebox{-.6ex}{$\begin{array}[b]{*{4}c}\cline{4-4}
&&&\lr{9}\\\cline{4-4}
&&\\\cline{1-3}
\lr{7}&\lr{8}&\lr{10}\\\cline{1-3}
\end{array}$}
}
\quad 
{\def\lr#1{\multicolumn{1}{|@{\hspace{.6ex}}c@{\hspace{.6ex}}|}{\raisebox{-.3ex}{$#1$}}}
\raisebox{-.6ex}{$\begin{array}[b]{*{4}c}\cline{4-4}
&&&\lr{10}\\\cline{4-4}
&&\\\cline{1-3}
\lr{7}&\lr{8}&\lr{9}\\\cline{1-3}
\end{array}$}
}
\quad 
{\def\lr#1{\multicolumn{1}{|@{\hspace{.6ex}}c@{\hspace{.6ex}}|}{\raisebox{-.3ex}{$#1$}}}
\raisebox{-.6ex}{$\begin{array}[b]{*{4}c}\cline{4-4}
&&&\lr{7}\\\cline{4-4}
&&\\\cline{1-3}
\lr{8}&\lr{9}&\lr{10}\\\cline{1-3}
\end{array}$}
}
\quad 
.\end{displaymath}
\label{example-lam-small}
\end{mexample}

\begin{remark}
This example was computed using SageMath \autocite{sagemath} code developed
in the preparation of \autocite{jaldevik-thesis}. The calculations in the examples to come, 
and figures and tables, were also produced using this code, which is available in the ancillary folder in the arXiv submission. We also made use of the symbolic computation
capabilities of Sympy \autocite{meurer2017sympy} and Maxima \autocite{maxima} to
simplify binomial expressions and to calculate limits of such expressions.
\end{remark}
\subsection{Blocking expansion for partition posets}
\label{sec:org98708f3}

\begin{definition}
Let \(\lambda=(\lambda_{1},\dots,\lambda_{k})\) be a partition and \(P_{\lambda}\) its cell poset. We define
\(\fla := e(P_{\lambda})\). When \(\mu \subset \lambda\), i.e. \(\mu\) is a subpartition of \(\lambda\),
we denote by \(\lamu\) the skew partition whose cells
are the set theoretic difference of the cells in \(\lambda\) and the cells in \(\mu\).
By definition, \(P_{\lamu}\) is the cell poset on \(\lamu\),  and \(\flamu = e(P_{{\lamu}})\).
\end{definition}

The quantities \(\fla\) and \(\flamu\) can be interpreted as
\begin{enumerate}
\item the number of SYT on \(\lambda\) (on \(\lamu\)),
\item the number of saturated chains in Young's lattice starting at \(\emptyset\) and ending
at \(\lambda\) (starting at \(\mu\) and ending at \(\lambda\)),
\item the number of lattice paths in \(\NN^{k}\) starting at the origin and ending in \(\lambda\),
staying inside the simplex \(x_{1} \ge x_{2} \ge \cdots \ge x_{k}\),
\item the \(\CC\)-vector space dimension of the Specht module \(S^{\lambda}\);
when \(\lambda\) ranges over all
partitions of weight \(n\) these modules form a complete list of irreducible representations
of \(S_{n}\), the symmetric group on \(n\) letters.
\end{enumerate}

There are several explicit formulae for \(\fla\) and \(\flamu\).
For our considerations, the most useful ones are
\begin{itemize}
\item the ``hook length formula'' for \(\fla\) by Frame, Robinson, and Thrall \autocite{hook-length-formula}
(see also \autocite{stanley-enum2}),
\item the Jacobi-Trudi-Aitken formula for \(\fla\) and \(\flamu\) \autocite{aitken}.
\end{itemize}

We will recall these formulae in the next subsection.  
\begin{remark}
In an earlier version of this manuscript we made use of the
``excited diagrams'' of Naruse, which
allows the computation of \(\flamu\) from hook-length data on \(\lamu\)
(see \autocite{naruse-excited} and also
\autocite{konvalinka2018bijectiveproofhooklengthformula}\autocite{Morales_2018}\autocite{Morales_2017}\autocite{morales2020hookformulasskewshapes}). Another option is to use lattice path
combinatorics, see for instance \autocite{lattice-path-enum}. Since we can make do with
just Jacobi-Aitken-Trudi, we have omitted this exposition. 
\end{remark}

For partition posets, Corollary \ref{cor-refine} becomes:
\begin{corollary}
Let \(\lambda\) be a partition, and \(P=P_{\lambda}\) the corresponding partition poset. Let \(\alpha, \beta \in P\)
be two incomparable elements. Put \(\mathcal{B} = \blocking{P}{\alpha}{\beta}\). Then
\begin{equation}\tag{BE}
\label{eq:e-part-refine}
e(P; \alpha < \beta) = \sum_{T \in \mathcal{B}} f^{T} f^{\lambda / (T \cup \set{\alpha}).}
\end{equation}
Consequently,
\begin{equation}
\label{eq:e-prob-refine}
\Prob(P; \alpha < \beta) =  \frac{e(P; \alpha < \beta)}{e(P)} = \sum_{T \in \mathcal{B}} f^{T} \frac{f^{\lambda / (T \cup \set{\alpha})}}{\fla}
\end{equation}
\label{cor-part-refine}
\end{corollary}

We call (\ref{eq:e-part-refine}) the \textbf{blocking-expansion} of \(e(P; \alpha < \beta)\).

\begin{mexample}
We continue example \ref{example-lam-small}, with \(P=P_\lambda\), \(\lambda=(4,3,3)\), \(\alpha=(2,3)\), \(\beta=(3,2)\).
We have that \(e(P) = \fla =\) \texttt{210}.
The ``blocking expansion'' for \(e(P; \alpha < \beta)\) is
\begin{displaymath}
f^{(4,2,1)}f^{(4,3,3) / (4,3,1)} + f^{(4,2)}f^{(4,3,3) / (4,3)} + f^{(3,2)}f^{(4,3,3) / (3,3)} + f^{(3,2,1)}f^{(4,3,3) / (3,3,1)}
.
\end{displaymath}
We can ``reduce'' some skew partitions \(L /M\) by shifting them
to the left to obtain \(\tilde{L} / \tilde{M}\) with \(f^{L / M} = f^{\tilde{L} / \tilde{M}}\).
This is the case when the entire first columns of the skew partition is empty
(see Lemma \ref{lemma-reduce} and example \ref{example-reduce}).
Doing this, we obtain
\begin{displaymath}
f^{(4,2,1)}f^{(3,2,2) / (3,2)} + f^{(4,2)}f^{(4,3,3) / (4,3)} + f^{(3,2)}f^{(4,3,3) / (3,3)} + f^{(3,2,1)}f^{(3,2,2) / (2,2)}
,
\end{displaymath}
which evaluates to
\begin{displaymath}
35 * 1 + 9 * 1 + 5 * 4 + 16 * 3
=
112.
\end{displaymath}
The probability \(\Prob(P; \alpha < \beta) = e(P; \alpha < \beta)/e(P)\) is therefore
8/15.
\end{mexample}
\subsection{The hook-length formula by Frame, Robinson, and Thrall}
\label{sec:org767484c}
\begin{theorem}[Hook-length formula]
The number of SYT of shape \(\lambda=(\lambda_{1}, \dots, \lambda_{k})\) is
\begin{equation}
\label{eq-HL}
\fla = \frac{|\lambda|!}{\prod_{c \in \lambda}{h_\lambda(c)}} \tag{HL}
\end{equation}
where \(h_\lambda(c)\) is the \textit{hook length}
of the cell \(c=(i,j)\), given by
the number of cells \(d = (s, t) \in \lambda\) such that
(\(i = s\)  and \(j \le t\)) or  (\(j = t\) and  \(i \le s\)).
\label{thm-hook-length}
\end{theorem}
The hook-length can also be expressed as
\(h_{\lambda}(c) = \lambda_{i} - j + \lambda'_{j} - i + 1\) where \(\lambda'\) is the \textbf{conjugate partition} of \(\lambda\).
\subsection{The Jacobi-Trudi-Aitken formula}
\label{sec:org4c41f5b}
\begin{theorem}[Jacobi-Trudi-Aitken]
The total number of standard skew tableaux of shape \(\lambda / \mu\) where
\(\lambda = (\lambda_1, \lambda_2,\dots,\lambda_n)\) and
\(\mu = (\mu_1,\mu_2,\dots,\mu_k)\) equals
\begin{equation}
\label{eq:aitken}
\flamu = |\lambda / \mu|!
\cdot \mathrm{det}(g(\lambda_i - \mu_j - i + j))_{i,j = 1}^n \tag{JTA}
\end{equation}
where
\begin{displaymath}
g(m) =
    \begin{cases}
        0 \text{ if } m < 0, \\
        \frac{1}{m!} \text{ otherwise.}
    \end{cases}
\end{displaymath}
\label{thm-aitken}
\end{theorem}
\subsection{Reduction}
\label{sec:org1e91e70}

\begin{definition}
Suppose that \(\mu \subset \lambda\),
\(\lambda = (\lambda_{1},\dots,\lambda_{r})\), \(\mu = (\mu_{1},\dots,\mu_{r})\),
with \(\mu_{r}=k>0\).
Define \(\tilde{\lambda}\) and  \(\tilde{\mu}\),
the ``reduction'' of the \(\lambda\) and of   \(\mu\),
by
\begin{equation}
\tilde{\lambda} =(\lambda_{1}-k,\dots,\lambda_{r}-k), \qquad
\tilde{\mu} =(\mu_{1}-k,\dots,\mu_{r}-k).
\end{equation}
Of course, this will introduce trailing zeroes in \(\tilde{\mu}\), these
are ignored.

We analogously define the reduction of the skew partition \(\lamu\) as \(\lamur\).
\label{def-reduced}
\end{definition}

Obviously \(| \lambda / \mu | = |\tilde{\lambda} / \tilde{\mu}|\).
It is also true that for any cell \(c \in \lamur\),
\(h_{\lambda}(c) = h_{\lar}(c)\)
since the empty cells to the left does not influence the hook-lengths.

Our motivation for introducing this concept is the following:
\begin{lemma}
With notations as above,
\(\flamu = \flamur\).
\label{lemma-reduce}
\end{lemma}
\begin{proof}
There is a bijection between standard tableaux on \(\lamu\) and on \(\lamur\),
since the empty cells to the left exerts no influence.
\end{proof}

\begin{mexample}
If \(\lambda = (4,3,2)\), \(\mu=(2,1,1)\) then
\(\tilde{\lambda} = (3,2,1)\), \(\tilde{\mu} = (1,0,0)=(1)\).

\phantomsection
\label{}
\begin{displaymath}
\begin{ytableau}
X & X & O & O \\
X & O & O \\
X & O
\end{ytableau}
 \qquad
\begin{ytableau}
X & O & O \\
O & O \\
O
\end{ytableau}
\end{displaymath}

The corresponding hook lengths (for the cells in \(\lamu\) and \(\lamur\)) are the same:
\phantomsection
\label{}
\begin{displaymath}
\begin{ytableau}
6 & 5 & 3 & 1 \\
4 & 3 & 1 \\
2 & 1
\end{ytableau}
 \qquad
\begin{ytableau}
5 & 3 & 1 \\
3 & 1 \\
1
\end{ytableau}
\end{displaymath}
\label{example-reduce}
\end{mexample}
\section{Partitions with two rows}
\label{sec:org23dbf24}
\subsection{Warmup: Catalan poset, \(\alpha\) and \(\beta\) touching}
\label{sec:orgbf84b38}
We start with the case that was treated in \autocite{jaldevik-thesis},
and which was also previously done in \autocite{olson20181}.
Let \(\lambda=(n,n)\), and put \(P=P_\lambda\). Let \(\alpha=(1,b+1)\), \(\beta=(2,b)\).
In order to adhere to the treatment in \autocite{jaldevik-thesis}, and beacuse in this
specific case the blocking expansion for \(\beta < \alpha\) is simpler than that for
\(\alpha < \beta\), we will calculate \(\Prob(P; \alpha < \beta)\) as \(1 -\Prob(P; \beta < \alpha)\),
and thus \(e(P; \alpha < \beta) = e(P) -e(P; \beta < \alpha)\).

The ``decorated Tableau'' will hence be as follows (shown here for \(b=3\)):

\ytableausetup{nosmalltableaux}
\begin{center}
\begin{ytableau}
*(yellow) F & *(yellow) F & *(yellow) F & *(green) 2 & *(cyan) o \\ 
*(yellow) F & *(yellow) F & *(red) 1 & *(cyan) o & *(cyan) o
\end{ytableau}
.\end{center}
\ytableausetup{smalltableaux}

There is no variable part, and the fixed part is the subpartition \((b,b-1)\).
The blocking expansion becomes
\begin{equation}
\label{eq-cat1}
  \begin{split}
    e(P; \beta < \alpha))
    & = f^{(b,b-1)} f^{(n,n) / (b,b)} \\
    & = f^{(b,b-1)} f^{(n-b,n-b) / (0,0)} \\
    & = f^{(b,b-1)} f^{(n-b,n-b)}.
  \end{split}
\end{equation}

It is well-known (see for instance \autocite{stanley-enum2}) that for \(\lambda =(n,n)\),
\(\fla = C_{n} = (n+1)^{-1}\binom{2n}{n}\), the famous Catalan number. It follows that
\begin{equation}
\label{eq-cat2}
  \begin{split}
  e(P; \beta < \alpha)) &= C_{b} C_{n-b,} \\
  \Prob(P; \beta < \alpha) &= \frac{C_{b} C_{n-b}}{C_{n}}.
  \end{split}
\end{equation}
This can be simplified. From the well-know recursion
\begin{equation}
\label{eq-cat2a}
C_{n+1} = \frac{2(2n+1)}{n+2} C_{n}, \quad C_{0}=1,
\end{equation}
we get that
\begin{equation}
\label{eq-cat2b}
C_{n-b} = 2^{-b} \prod_{k=0}^{b-1}\frac{n - k + 1}{2(n - k) - 1}C_{n}.
\end{equation}
This allows us to rewrite (\ref{eq-cat2}) as
\begin{equation}
\label{eq-cat3}
  \Prob({P};  \beta < \alpha) = \frac{C_{b} C_{n-b}}{C_{n}}
   =  2^{-b} C_{b} \prod_{k=0}^{b-1}\frac{n-k+1}{2(n - k) - 1},
\end{equation}
so
\begin{equation}
\label{eq-cat3inv}
  \Prob({P};  \alpha < \beta) = 1- \Prob({P};  \beta < \alpha)  =
  1- 2^{-b} C_{b}  \prod_{k=0}^{b-1}\frac{n-k+1}{2(n - k) -1}.
\end{equation}
We can calculate the limit probabilities for an ``infinite Catalan strip'' as
\begin{equation}
\label{eq-cat3lim}
  \begin{split}
    \lim_{n \to \infty}\Prob(P_{{(n,n)}}; \beta < \alpha) & =
            \lim_{n \to \infty} 2^{-b}  \prod_{k=0}^{b-1}{\frac{n-k+1}{2(n-k)-1}} C_b \\
            & = 2^{-b} C_b \prod_{k=0}^{b-1}{\lim_{n \to \infty}\frac{n-k+1}{2(n-k)-1}} \\
            & =2^{-b} C_b \prod_{k=0}^{b-1}{\lim_{n \to \infty}\frac{1 - \frac{k-1}{n}}{2 - \frac{2k+1}{n}}} \\
            & = 2^{-b} C_b \prod_{k=0}^{b-1}{\frac{1}{2}} \\
            & = 2^{-b} C_b \left( \frac{1}{2} \right)^b \\
            &= 4^{-b} C_b,
        \end{split}
\end{equation}
thus
\begin{equation}
\label{eq-cat3liminv}
    \lim_{n \to \infty}\Prob(P_{{(n,n)}}; \alpha < \beta)  = 1- 4^{-b} C_b.
\end{equation}

We plot \(\Prob(P_{{(n,n)}}; \alpha < \beta)\) for  \(2 \le b \le 6, b < n < 50\) in Figure \ref{limprob21}.
\begin{figure}[htbp]
\centering
\includegraphics[width=.9\linewidth]{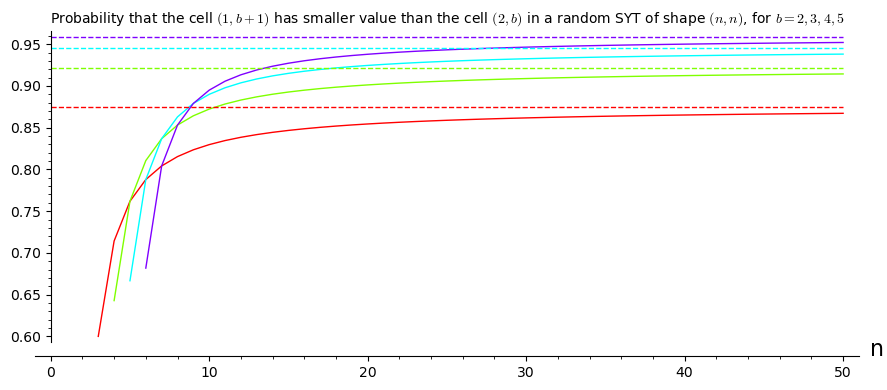}
\caption{\label{limprob21}Limit probabilities for \(\Prob(P_{{(n,n)}}; (1,b+1) < (2,b))\)}
\end{figure}
\subsection{Blocking expansion for partitions with two rows}
\label{sec:org81ce1d0}
When \(\lambda = (\lambda_1,\lambda_2)\), with \(\lambda_{2} > 0\), and \(\alpha,\beta\) incomparable elements in \(P=P_{\lambda}\),
we can assume that \(\alpha=(1,a)\), \(\beta=(2,b)\), with \(a > b\). The structure of
the blocking ideals \(\blocking{P_{\lambda}}{\alpha}{\beta}\) (and \(\blocking{P_{\lambda}}{\beta}{\alpha}\)) becomes very simple,
and the \(f^{T}\) and \(f^{\lambda / \mu}\), with \(\mu = T \cup \{\alpha\}\), are easy to calculate.

The situation becomes especially nice if we assume that \(\lambda_{2} \gg 0\), so that the right borders of
\(\lambda\) do not interact with \(\alpha\) or \(\beta\), like here:
\begin{displaymath}
\ydiagram[*(green) \bullet]
{5+1, 2+1}
*[*(white)]{10,7}
\end{displaymath}

It will be the situation we are in if we want
to consider limit, as \(\lambda=(\lambda_{1},\lambda_{2}) \to (\infty,\infty)\) in such a way that \(\lambda_{1} - \lambda_{2}\)
remains bounded, of the poset probability
\begin{equation}
\label{eq-swap-lim}
  \begin{split}
    \lim_{\lambda \to (\infty,\infty)}  \Prob(P_{\lambda}; \alpha < \beta)
    & = \lim_{\lambda \to (\infty,\infty)} \frac{e(P_{\lambda}; \alpha < \beta)}{e(P_{\lambda})} \\
    & = \lim_{\lambda \to (\infty,\infty)} \sum_{T \in \blocking{P}{\alpha}{\beta}} f^{T} \frac{f^{\lambda / (T \cup \set{\alpha})}}{\fla} \\
    &  = \sum_{T \in \blocking{P}{\alpha}{\beta}} f^{T} \lim_{\lambda \to (\infty,\infty) } \frac{f^{\lambda / (T \cup \set{\alpha})}}{\fla}
  \end{split}
\end{equation}
Here, we used equation (\ref{eq:e-prob-refine}).

\begin{mexample}
We study blocking ideals in \(P_\lambda\) when \(\lambda\)
and \((\alpha,\beta)\) are respectively

\begin{equation}
\begin{matrix}
 \lambda & \alpha & \mu & Diagram \\
\hline (7,6) & (1,5) & (2,3) & \begin{ytableau}*(white) & *(white) & *(white) & *(white) & *(red) \alpha  & *(white) & *(white) \\*(white) & *(white) & *(green) \beta  & *(white) & *(white) & *(white)\end{ytableau} \\
 (7,6) & (1,5) & (2,4) & \begin{ytableau} *(white) & *(white) & *(white) & *(white) & *(red) \alpha  & *(white) & *(white) \\*(white) & *(white) & *(white) & *(green) \beta  & *(white) & *(white) \end{ytableau} \\
 (7,6) & (1,5) & (2,1) & \begin{ytableau} *(white) & *(white) & *(white) & *(white) & *(red) \alpha  & *(white) & *(white) \\ *(green) \beta  & *(white) &  *(white) & *(white) & *(white) & *(white) \end{ytableau} \\
\end{matrix}
\label{tab-twr}
\end{equation}

For these cases, we study the situation with \(\alpha\) preceding \(\beta\),
calculating the ``decorated tableau''; the result is shown in Table \ref{tab:tworowdec}.

\phantomsection
\label{table-tworows-abeforeb-generic}
\begin{table}[h]
\begin{adjustwidth}{-0.1cm}{-0.1cm}
\label{table:tworows}
\caption{Partitions with two rows, \(\alpha\) in first row}
\label{tab:tworowdec}
\begin{center}
\begin{tabularx}{420pt}{|X|X|X|X|} \hline
\(\lambda, \alpha, \beta\) &
(7,6) , (1, 5) , (2, 3) & (7,6) , (1, 5) , (2, 4) &  (7,6) , (1, 5) , (2, 1)  \\ \hline Tableau & \begin{ytableau}
*(yellow) F & *(yellow) F & *(yellow) F & *(yellow) F & *(red) 1 & *(cyan) o & *(cyan) o \\ 
*(gray) V & *(gray) V & *(green) 2 & *(cyan) o & *(cyan) o & *(cyan) o
\end{ytableau} & \begin{ytableau}
*(yellow) F & *(yellow) F & *(yellow) F & *(yellow) F & *(red) 1 & *(cyan) o & *(cyan) o \\ 
*(gray) V & *(gray) V & *(gray) V & *(green) 2 & *(cyan) o & *(cyan) o
\end{ytableau} & \begin{ytableau}
*(yellow) F & *(yellow) F & *(yellow) F & *(yellow) F & *(red) 1 & *(cyan) o & *(cyan) o \\ 
*(green) 2 & *(cyan) o & *(cyan) o & *(cyan) o & *(cyan) o & *(cyan) o
\end{ytableau} 
  \\ \hline
Fixed part &
 (4) & (4) & (4) 
  \\ \hline
Complete ideal &
 (4,2) & (4,3) & (4) 
  \\ \hline
Blocking expansion &
 \begin{math}
f^{(4)}f^{(7,6) / (5)}  +  f^{(4,2)}f^{(7,6) / (5,2)}  +  f^{(4,1)}f^{(7,6) / (5,1)}
\end{math} & \begin{math}
f^{(4)}f^{(7,6) / (5)}  +  f^{(4,2)}f^{(7,6) / (5,2)}  +  f^{(4,3)}f^{(7,6) / (5,3)}  +  f^{(4,1)}f^{(7,6) / (5,1)}
\end{math} & \begin{math}
f^{(4)}f^{(7,6) / (5)}
\end{math} 
  \\ \hline
Reduced blocking expansion &
 \begin{math}
f^{(4)}f^{(7,6) / (5)}  +  f^{(4,2)}f^{(5,4) / (3)}  +  f^{(4,1)}f^{(6,5) / (4)}
\end{math} & \begin{math}
f^{(4)}f^{(7,6) / (5)}  +  f^{(4,2)}f^{(5,4) / (3)}  +  f^{(4,3)}f^{(4,3) / (2)}  +  f^{(4,1)}f^{(6,5) / (4)}
\end{math} & \begin{math}
f^{(4)}f^{(7,6) / (5)}
\end{math} 
  \\ \hline

\end{tabularx}
\end{center}
\end{adjustwidth}
\end{table}
\end{mexample}

Guided by these examples, we state the following theorem for \(P=P_{\lambda}\) with \(\lambda=(\lambda_{1},\lambda_{2})\).

\begin{theorem}
Let \(\lambda=(\lambda_{1},\lambda_{2})\) be a partition with two rows, and let \(P=P_{\lambda}\) be its cell poset.
Suppose that \(\alpha,\beta \in P_{\lambda}\) are incomparable. The \(\alpha\) and \(\beta\) belong to different rows,
and there are two cases.
When  \(\alpha=(1,a)\), \(\beta=(2,b)\) with \(1\le b<a \le \lambda_{1}\), then
\begin{align}
\label{eq-abeforebexp}
 A_{{P}}(\alpha) & = (a-1) \\
 G_{{P}}(\alpha,\beta) & = \{(2,j) \mid 1 \le j < b\} \\
 D_{{P}}(\alpha,\beta) & = (a-1,b-1) \\
 \blocking{P}{\alpha}{\beta} & = \{(a -1,t) \mid 0 \le t \le  b-1\} \\
 e(P;\alpha < \beta) & = \sum_{t=0}^{b-1} f^{(a-1,t)} f^{\lambda / (a,t)} \\
 & = f^{(a-1)} f^{\lambda / (a)} + \sum_{t=1}^{b-1} f^{(a-1,t)} f^{(\lambda_{1}-t,\lambda_{2}-t) / (a-t)}
 \end{align}

If \(\alpha=(2,a)\), \(\beta=(1,b)\) with \(1\le a < b \le \lambda_{1}\), then
\begin{align}
\label{eq-bbeforeaexp}
    A_{{P}}(\alpha)  &= (a,a-1) \\
    G_{{P}}(\alpha,\beta) &= \{(1,j) \mid a < j < b\} \\
    D_{{P}}(\alpha,\beta) &= (b-1,a-1) \\
    \blocking{P}{\alpha}{\beta} &= \{(t,a-1) \mid a \le t \le b-1\} \\
    e(P;\alpha < \beta) &= \sum_{t=a}^{b-1} f^{(t,a-1)} f^{\lambda / (t,a) }
    = \sum_{t=a}^{b-1} f^{(t,a-1)} f^{(\lambda_{1}-a,\lambda_{2}-a) / (t-a) }
\end{align}    
\label{thm-tworows}
\end{theorem}
\begin{proof}
We first assume that \(\alpha=(1,a)\), \(\beta=(2,b)\), with \(1\le b<a \le \lambda_{1}\).
To facilitate the visualisation of the steps taken in the proof, here is the decorated tableau when \(\lambda=(7,6)\), \(a=5\), \(b=3\):
\ytableausetup{nosmalltableaux}
\begin{center}
\begin{ytableau}
*(yellow) F & *(yellow) F & *(yellow) F & *(yellow) F & *(red) a & *(cyan) o & *(cyan) o \\
*(gray) V & *(gray) V & *(green) b & *(cyan) o & *(cyan) o & *(cyan) o
\end{ytableau}.
\end{center}
\ytableausetup{smalltableaux}

It is immediate that \(A_{P}(\alpha) =\{(1,i) \mid 1 \le i \le a-1\}\) which we can represent as the subpartition
\((a-1) \subset \lambda\).
Since \(G_{{P}}(\alpha,\beta)\) is the set of cells that are
\(\not \ge \alpha\),
\(\not \ge \beta\),
\(\not \le \alpha\),
they must reside on the second row, to the left of \(\beta\), hence
\(G_{P}(\alpha,\beta) =\{(2,j) \mid 1 \le j \le b - 1\).
The complete ideal \(D_{p}(\alpha,\beta)\)
is comprised of the cells that are
\(\not \ge \alpha\),
\(\not \ge \beta\),
so from the first row, it chooses precisely \(A_{p}(\alpha)\), and from the second row,
\(D_{P}(\alpha,\beta)\), hence
\[
D_{{P}}(\alpha,\beta) = \{(1,i) \mid 1 \le i \le a-1\} \cup \{(2,j) \mid 1 \le j \le b-1\} = (a-1,b-1),
\]
with the convention that \(b-1\) may be zero.
A blocking ideal \(T\) contains all of \(A_{p}(\alpha)\) and an order ideal subset of
\(G_{P}(\alpha,\beta)\), so it is of the form \((a-1,t)\) with \(0 \le t \le b-1\).
For this \(T\), \(T \cup \{\alpha\} = (a,t)\), so the blocking expansion becomes
\[
e(P;\alpha < \beta) = \sum_{t=0}^{b-1} f^{(a-1,t)} f^{\lambda / (a,t).}
\]
For \(t \ge 1\), the skew partition \(\lambda / (a,t)\) has \(t\) empty columns to the left,
and can be ``reduced'' to \((\lambda_{1}-t,\lambda_{2}-t) / (a-t)\) since
\(f^{\lambda / (a,t)} = f^{(\lambda_{1}-t,\lambda_{2}-t) / (a-t)}\); the reduced expansion is thus
\[ e(P;\alpha < \beta) = f^{(a-1)} f^{\lambda / (a)} + \sum_{t=1}^{b-1} f^{(a-1,t)} f^{(\lambda_{1}-t,\lambda_{2}-t) / (a-t)}.\]

The case \(\alpha=(2,a)\), \(\beta=(1,b)\), \(1\le a < b \le \lambda_{1}\) is proved similarly.
\end{proof}
\subsection{Number of SYT on partitions and skew-partitions with two rows}
\label{sec:org8b14970}
To find \(\Prob(P_{\lambda}; \alpha < \beta) = \frac{e(P_{\lambda}(\alpha,\beta)}{e(P_{\lambda})}\) for partition posets on partitions
with two rows, we start by  calculating \(\fla = e(P_{\lambda})\) for
\(\lambda = (\lambda_{1},\lambda_{1})\). Using Jacobi-Trudi-Aitken, we have:
\begin{lemma}
When \(\lambda = (\lambda_{1},\lambda_{2})\) with \(0 < \lambda_{2} \le \lambda_{1}\), 
then
\begin{equation}
\label{eq-fla-tworows}
  f^{\lambda} = (\lambda_{1} + \lambda_{2})!
       \begin{vmatrix}
          \frac{1}{\lambda_{1}!} & \frac{1}{(\lambda_{1} + 1)!} \\
          \frac{1}{(\lambda_{2} - 1)!} & \frac{1}{\lambda_{2}!} \\
       \end{vmatrix} 
     = \frac{(\lambda_{1} +  \lambda_{2})!(1 + \lambda_{1} - \lambda_{2})}{(\lambda_{1}+1)!\lambda_{2}!}
\end{equation}
\label{lemma-fl-tworows}
\end{lemma}

Next, since terms of the form \(\flamu\) occur in the blocking expansion of
\(e(P_{\lambda};  \alpha < \beta)\), we calculate those, for two-row skew partitions. We will
also be interested in the limit of \(\Prob({P_{\lambda}}; \alpha < \beta)\) as
\((\lambda_{1},\lambda_{2}) \to (\infty,\infty)\) in such a way that \(\lambda_{1} - \lambda_{2}\) remains bounded. We therefore
compute the corresponding limits of \(\frac{\flamu}{\fla}\).

\begin{lemma}
Let \(\lambda=(\lambda_{1},\lambda_{2})\), \(\mu=(\mu_{1})\),
\(\tilde{\lambda}=(\lambda_{1}+k,\lambda_{2}+k)\),
\(\tilde{\mu}=(\mu_{1}+k,k)\).
Then
\begin{equation}
\label{eq-flamutw}
f^{\lambda / \mu}  = f^{\tilde{\lambda} / \tilde{\mu}}
=
|\lambda / \mu|!
\left \lvert
\begin{matrix}
g(\lambda_{1} - \mu_{1}) & g(\lambda_{1} + 1) \\
g(\lambda_{2} - \mu_{{1}} -1) & g(\lambda_{2})
\end{matrix}
\right \rvert
\end{equation}
If \(\mu_{1} \ge \lambda_{2}\) then this quantity is
\begin{equation}
\label{eq-flamutw-ce}
  \frac{(\lambda_{1} + \lambda_{2} - \mu_{1})!}{(\lambda_{1} - \mu_{1})! \lambda_{2}!}
  = \binom{\lambda_{1} + \lambda_{2} - \mu_{1}}{\lambda_{2}}
\end{equation}
and if \(\mu_{1} < \lambda_{2}\) it is
\begin{equation}
\label{eq-flamutw-cg}
(\lambda_{1} + \lambda_{2} - \mu_{1})!
\left(
\frac{1}{(\lambda_{1} - \mu_{1})!} \frac{1}{\lambda_{1}!}
-
\frac{1}{(\lambda_{1}+1)!}  \frac{1}{(\lambda_{2} - \mu_{1} -1)!}
\right)
\end{equation}

Consequently,
\begin{equation}
\label{eq-probtw-c}
  \begin{split}
    \frac{\flamu}{\fla}
    & =
      \frac{|\lambda|!
      \begin{vmatrix}
        g(\lambda_{1}) & g(\lambda_{1} + 1) \\
        g(\lambda_{2} -1) & g(\lambda_{2})
      \end{vmatrix}
      }
      {|\lambda / \mu|!
      \begin{vmatrix}
        g(\lambda_{1} - \mu_{1}) & g(\lambda_{1} + 1) \\
        g(\lambda_{2} - \mu_{{1}} -1) & g(\lambda_{2})
      \end{vmatrix}
      } \\
  \end{split}
 \end{equation}
which for \(\mu_{1} \ge \lambda_{2}\) is
\begin{equation}
\label{eq-qtw-ce}
\frac{(\lambda_{1} + \lambda_{2} - \mu_{1})!}{(\lambda_{1} - \mu_{1})!}
\frac{(\lambda_{1}+1)!}{1 + \lambda_{1} - \lambda_{2}}
\end{equation}
and for \(\mu_{1} < \lambda_{2}\) is
\phantomsection
\label{eq-flq-sym}
\begin{equation}
 \frac{{\left({\left({\lambda_1} + 1\right)} {\lambda_1}! \left({\lambda_2} - {\mu_1} - 1\right)! - {\lambda_2} \left({\lambda_1} - {\mu_1}\right)! \left({\lambda_2} - 1\right)!\right)} \left({\lambda_1} + {\lambda_2} - {\mu_1}\right)!}{{\left({\lambda_1} - {\lambda_2} + 1\right)} \left({\lambda_1} + {\lambda_2}\right)! \left({\lambda_1} - {\mu_1}\right)! \left({\lambda_2} - {\mu_1} - 1\right)!}
 \end{equation}
When \((\lambda_{1},\lambda_{2}) \to (+\infty, +\infty)\) with \(\lambda_{1} - \lambda_{2}\) bounded,
this last expression, and thus \(\frac{\flamu}{\fla}\), tends to the limit
\phantomsection
\label{eq-flq-eqlim}
\begin{equation}
 \frac{{\mu_1} + 1}{2^{{\mu_1}}}
 \end{equation}
\label{lemma-flm-tworows}
\end{lemma}
\begin{proof}
The identity (\ref{eq-flamutw}) follows from Theorem \ref{thm-aitken} and Lemma \ref{lemma-fl-tworows};
(\ref{eq-flamutw-ce}) and  (\ref{eq-flamutw-cg}) are immediate consequences. Using Lemma \ref{lemma-fl-tworows}
yields the next three equations. The last limit was found using the computer algebra package Sympy \autocite{meurer2017sympy}.
\end{proof}
\subsection{Number of SYT with placing a given cell before another given cell, for partitions  with two rows}
\label{sec:org4d192ff}
We continue to study the case \(\lambda=(\lambda_{1}, \lambda_{2})\), \(\alpha, \beta \in P_{\lambda}\) incomparable. As noted earlier,
we must have that \(\alpha\) and \(\beta\) are situated in different rows.
Since
\begin{equation}
  e(P_{\lambda}; \alpha < \beta) + e(P_{\lambda}; \beta < \alpha)
  = e(P_{\lambda})
  = \fla
\end{equation}
we can assume that \(\alpha\) is in the first row, and \(\beta\) in the second row,
so \(\alpha=(1,a)\), \(\beta=(2,b)\), \(1 \le b <a \le \lambda_{1}\).
\subsubsection{The cell in the first row precedes the cell in the second row, which is situated in the beginning of the row}
\label{sec:org75cbd31}
Before tackling the general case, we study what happens when \(\alpha=(1,a)\), \(\beta=(2,1)\), \(\lambda=(\lambda_{1},\lambda_{2})\), \(\alpha\) precedes \(\beta\).
It is almost as simple as the ``warmup'' case, since the ``blocking expansion'' will have only one term. We illustrate the particular case \(\lambda=(10,7)\), \(a=5\).
\[\ydiagram[*(green) \bullet]
    {4+1, 0+1}
    *[*(white)]{10,7}.
\]

\begin{theorem}
Let \(a>1\) be an integer, \(\alpha=(1,a)\), \(\beta=(2,1)\), \(\lambda=(\lambda_{1},\lambda_{2})\).
Then
\begin{equation}
\label{eq:lam2b1expansion}
  \begin{split}
    e(P_{\lambda}; \alpha < \beta) & = f^{(a-1)} f^{\lambda / (a)} = f^{\lambda / (a)} \\
    \Prob(P_{\lambda}; \alpha < \beta)  & = \frac{e(P_{\lambda}; \alpha < \beta)}{e(P_{\lambda})} = \frac{f^{\lambda / (a)}}{\fla}.
  \end{split}
\end{equation}
For the ``special case'' when \(\lambda_{2} \le a\) we have that
\begin{equation}
\label{eq:lam2b1value-exceptional}
  \begin{split}
    f^{\lambda / (a)} & = \frac{(\lambda_{1 }+ \lambda_{2} - a)!}{(\lambda_1 - a)! \lambda_2!} \\
    \frac{f^{\lambda / (a)}}{\fla} &  = \frac{(\lambda_{1} + \lambda_{2} -a)! (\lambda_{1} + 1)!}{(\lambda_{1} - a)!(\lambda_{1} + \lambda_{2})! (1 + \lambda_1 - \lambda_2)}.
  \end{split}
\end{equation}
whereas for the ``generic case'' \(\lambda_{2} > a\) it holds that
\begin{equation}\label{eq:f3}
\begin{split}
f^{\lambda / (a)} & = 
-{\left(\frac{1}{\left(-a + {\lambda_2} - 1\right)! \left({\lambda_1} + 1\right)!} - \frac{1}{\left(-a + {\lambda_1}\right)! {\lambda_2}!}\right)} \left(-a + {\lambda_1} + {\lambda_2}\right)!
 \\  \frac{f^{\lambda / (a)}}{\fla} &  = 
\frac{{\left({\left({\lambda_1} + 1\right)} \left(-a + {\lambda_2} - 1\right)! {\lambda_1}! - {\lambda_2} \left(-a + {\lambda_1}\right)! \left({\lambda_2} - 1\right)!\right)} \left(-a + {\lambda_1} + {\lambda_2}\right)!}{{\left({\lambda_1} - {\lambda_2} + 1\right)} \left(-a + {\lambda_1}\right)! \left(-a + {\lambda_2} - 1\right)! \left({\lambda_1} + {\lambda_2}\right)!}
\end{split}
\end{equation}
\label{thm-b-br}
\end{theorem}
\begin{proof}
Apply Theorem \ref{thm-tworows} and Lemma \ref{lemma-flm-tworows}.
\end{proof}

\begin{mexample}
When \(a=2\), the probability \(c(\lambda_{1},\lambda_{2}) = \Prob(P_{(\lambda_{1},\lambda_{2})}; (1,2) < (2,1))\)
is given by specialising \eqref{eq:f3}, yielding

\begin{equation}
\label{eq:f3a2}
 \frac{{\lambda_1}^{2} + {\left({\lambda_1} - 2\right)} {\lambda_2} + {\lambda_2}^{2} - {\lambda_1}}{{\lambda_1}^{2} + {\left(2 \, {\lambda_1} - 1\right)} {\lambda_2} + {\lambda_2}^{2} - {\lambda_1}} 
 \end{equation}
For small \(\lambda_{1} \ge \lambda_{2}\), \(c(\lambda_{1},\lambda_{2})\) varies as follows (\(\lambda_1\) indexes rows, \(\lambda_{2}\) columns):
\begin{equation}
 \left(\begin{array}{rrrrrrrrrr}
\frac{1}{2} & 0 & 0 & 0 & 0 & 0 & 0 & 0 & 0 & 0 \\
\frac{3}{5} & \frac{3}{5} & 0 & 0 & 0 & 0 & 0 & 0 & 0 & 0 \\
\frac{2}{3} & \frac{9}{14} & \frac{9}{14} & 0 & 0 & 0 & 0 & 0 & 0 & 0 \\
\frac{5}{7} & \frac{19}{28} & \frac{2}{3} & \frac{2}{3} & 0 & 0 & 0 & 0 & 0 & 0 \\
\frac{3}{4} & \frac{17}{24} & \frac{31}{45} & \frac{15}{22} & \frac{15}{22} & 0 & 0 & 0 & 0 & 0 \\
\frac{7}{9} & \frac{11}{15} & \frac{39}{55} & \frac{23}{33} & \frac{9}{13} & \frac{9}{13} & 0 & 0 & 0 & 0 \\
\frac{4}{5} & \frac{83}{110} & \frac{8}{11} & \frac{37}{52} & \frac{64}{91} & \frac{7}{10} & \frac{7}{10} & 0 & 0 & 0 \\
\frac{9}{11} & \frac{17}{22} & \frac{29}{39} & \frac{66}{91} & \frac{5}{7} & \frac{17}{24} & \frac{12}{17} & \frac{12}{17} & 0 & 0 \\
\frac{5}{6} & \frac{41}{52} & \frac{69}{91} & \frac{31}{42} & \frac{29}{40} & \frac{195}{272} & \frac{109}{153} & \frac{27}{38} & \frac{27}{38} & 0 \\
\frac{11}{13} & \frac{73}{91} & \frac{27}{35} & \frac{3}{4} & \frac{25}{34} & \frac{37}{51} & \frac{41}{57} & \frac{68}{95} & \frac{5}{7} & \frac{5}{7}
\end{array}\right) 
 \end{equation}

This table is row increasing but column decreasing;
indeed, we can verify that 
\begin{equation}
\label{eq:u1}
 c(\lambda_1 +1,\lambda_2) - c(\lambda_1,\lambda_2) = \frac{{\left(\lambda_{1} - \lambda_{2} + 3\right)} \lambda_{2}}{{\left(\lambda_{1} + \lambda_{2} + 1\right)} {\left(\lambda_{1} + \lambda_{2}\right)} {\left(\lambda_{1} + \lambda_{2} - 1\right)}} > 0 
 \end{equation}
and that
\begin{equation}
\label{eq:u2}
 c(\lambda_1,\lambda_2+1) - c(\lambda_1,\lambda_2) = -\frac{{\left(\lambda_{1} - \lambda_{2} - 1\right)} {\left(\lambda_{1} + 1\right)}}{{\left(\lambda_{1} + \lambda_{2} + 1\right)} {\left(\lambda_{1} + \lambda_{2}\right)} {\left(\lambda_{1} + \lambda_{2} - 1\right)}} \le 0 
 \end{equation}
The limits of \(c(\lambda_{1},\lambda_{2})\) for a a fixed \(\lambda_{2}\), i.e. down a fixed column,
is \(\lim_{{t \to \infty}}c(t,\lambda_{2}) = 1\).
Down the main diagonal (or parallel to it) the limit is
\(\lim_{{t \to \infty}}c(t+r,t) = 3/4\). This can be verified using formal symbolic methods;
we used SymPy \autocite{meurer2017sympy}. More generally, we conjecture that
for \(p \ge q \ge 1\) coprime integers,
\begin{equation}
\label{a2-conj}
\lim_{t \to \infty} c(tp, tq) = \frac{p^2 + pq + q^2}{(p+q)^2}.
\end{equation}
This means that the limit of \(c\) along a line with rational slope
can take \textbf{any} rational value in the interval \([3/4,1]\).
\label{example-a2}
\end{mexample}

The  probabilities \(\Prob(P_{(\lambda_{1},\lambda_{2})}; (1,a) < (2,b))\) for \(a=2,3,4,5\), \(b=1\)
are shown below.

\(a=3, b=1\)
\begin{equation}
 \left(\begin{array}{rrrrrrrrrr}
\frac{1}{5} & 0 & 0 & 0 & 0 & 0 & 0 & 0 & 0 & 0 \\
\frac{2}{7} & \frac{2}{7} & 0 & 0 & 0 & 0 & 0 & 0 & 0 & 0 \\
\frac{5}{14} & \frac{1}{3} & \frac{1}{3} & 0 & 0 & 0 & 0 & 0 & 0 & 0 \\
\frac{5}{12} & \frac{17}{45} & \frac{4}{11} & \frac{4}{11} & 0 & 0 & 0 & 0 & 0 & 0 \\
\frac{7}{15} & \frac{23}{55} & \frac{13}{33} & \frac{5}{13} & \frac{5}{13} & 0 & 0 & 0 & 0 & 0 \\
\frac{28}{55} & \frac{5}{11} & \frac{11}{26} & \frac{37}{91} & \frac{2}{5} & \frac{2}{5} & 0 & 0 & 0 & 0 \\
\frac{6}{11} & \frac{19}{39} & \frac{41}{91} & \frac{3}{7} & \frac{5}{12} & \frac{7}{17} & \frac{7}{17} & 0 & 0 & 0 \\
\frac{15}{26} & \frac{47}{91} & \frac{10}{21} & \frac{9}{20} & \frac{59}{136} & \frac{65}{153} & \frac{8}{19} & \frac{8}{19} & 0 & 0 \\
\frac{55}{91} & \frac{19}{35} & \frac{1}{2} & \frac{8}{17} & \frac{23}{51} & \frac{25}{57} & \frac{41}{95} & \frac{3}{7} & \frac{3}{7} & 0 \\
\frac{22}{35} & \frac{17}{30} & \frac{71}{136} & \frac{25}{51} & \frac{80}{171} & \frac{43}{95} & \frac{31}{70} & \frac{101}{231} & \frac{10}{23} & \frac{10}{23}
\end{array}\right) 
 \end{equation}

\(a=4, b=1\)
\begin{equation}
 \left(\begin{array}{rrrrrrrrrr}
\frac{1}{14} & 0 & 0 & 0 & 0 & 0 & 0 & 0 & 0 & 0 \\
\frac{5}{42} & \frac{5}{42} & 0 & 0 & 0 & 0 & 0 & 0 & 0 & 0 \\
\frac{1}{6} & \frac{5}{33} & \frac{5}{33} & 0 & 0 & 0 & 0 & 0 & 0 & 0 \\
\frac{7}{33} & \frac{5}{27} & \frac{25}{143} & \frac{25}{143} & 0 & 0 & 0 & 0 & 0 & 0 \\
\frac{14}{55} & \frac{125}{572} & \frac{200}{1001} & \frac{5}{26} & \frac{5}{26} & 0 & 0 & 0 & 0 & 0 \\
\frac{42}{143} & \frac{251}{1001} & \frac{41}{182} & \frac{11}{52} & \frac{7}{34} & \frac{7}{34} & 0 & 0 & 0 & 0 \\
\frac{30}{91} & \frac{461}{1638} & \frac{114}{455} & \frac{63}{272} & \frac{203}{918} & \frac{70}{323} & \frac{70}{323} & 0 & 0 & 0 \\
\frac{33}{91} & \frac{113}{364} & \frac{131}{476} & \frac{257}{1020} & \frac{230}{969} & \frac{74}{323} & \frac{30}{133} & \frac{30}{133} & 0 & 0 \\
\frac{11}{28} & \frac{643}{1904} & \frac{61}{204} & \frac{1055}{3876} & \frac{1231}{4845} & \frac{129}{532} & \frac{345}{1463} & \frac{75}{322} & \frac{75}{322} & 0 \\
\frac{143}{340} & \frac{667}{1836} & \frac{2495}{7752} & \frac{283}{969} & \frac{139}{513} & \frac{268}{1045} & \frac{125}{506} & \frac{50}{207} & \frac{11}{46} & \frac{11}{46}
\end{array}\right) 
 \end{equation}

\(a=5, b=1\)
\begin{equation}
 \left(\begin{array}{rrrrrrrrrr}
\frac{1}{42} & 0 & 0 & 0 & 0 & 0 & 0 & 0 & 0 & 0 \\
\frac{1}{22} & \frac{1}{22} & 0 & 0 & 0 & 0 & 0 & 0 & 0 & 0 \\
\frac{7}{99} & \frac{9}{143} & \frac{9}{143} & 0 & 0 & 0 & 0 & 0 & 0 & 0 \\
\frac{14}{143} & \frac{83}{1001} & \frac{1}{13} & \frac{1}{13} & 0 & 0 & 0 & 0 & 0 & 0 \\
\frac{18}{143} & \frac{19}{182} & \frac{29}{312} & \frac{3}{34} & \frac{3}{34} & 0 & 0 & 0 & 0 & 0 \\
\frac{2}{13} & \frac{461}{3640} & \frac{15}{136} & \frac{31}{306} & \frac{63}{646} & \frac{63}{646} & 0 & 0 & 0 & 0 \\
\frac{33}{182} & \frac{71}{476} & \frac{131}{1020} & \frac{112}{969} & \frac{35}{323} & \frac{2}{19} & \frac{2}{19} & 0 & 0 & 0 \\
\frac{99}{476} & \frac{35}{204} & \frac{1709}{11628} & \frac{211}{1615} & \frac{16}{133} & \frac{502}{4389} & \frac{18}{161} & \frac{18}{161} & 0 & 0 \\
\frac{143}{612} & \frac{79}{408} & \frac{107}{646} & \frac{25}{171} & \frac{139}{1045} & \frac{63}{506} & \frac{11}{92} & \frac{27}{230} & \frac{27}{230} & 0 \\
\frac{1001}{3876} & \frac{139}{646} & \frac{7}{38} & \frac{237}{1463} & \frac{37}{253} & \frac{137}{1012} & \frac{59}{460} & \frac{371}{2990} & \frac{11}{90} & \frac{11}{90}
\end{array}\right) 
 \end{equation}

Pictorially, the  probabilities are displayed in Figure \ref{fig:Mat-tworows-b1-export}.

\begin{figure}[htbp]
\centering
\includegraphics[height=1.0\textheight]{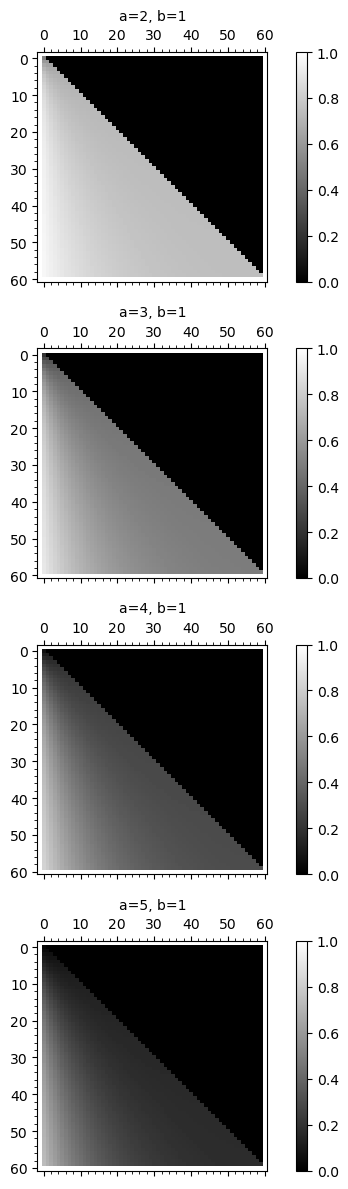}
\caption{\label{fig:Mat-tworows-b1-export}Probabilities \(\Prob(P_{(\lambda_{1},\lambda_{2})}; (1,a) < (2,b))\)  for two-row partitions, a=2,3,4,5 and  b=1, rows indexed by \(\lambda_1\), columns by \(\lambda_{2}\)}
\end{figure}

\begin{theorem}
Assume as before that \(\alpha=(1,a)\), \(\beta=(2,1)\), \(\lambda = (\lambda_{1},\lambda_{2})\).
When \((\lambda_{1},\lambda_{2}) \to (\infty,\infty)\), in such a way that \(\lambda_{1} - \lambda_{2}\) remains bounded, then \(\Prob(P_{\lambda}; \alpha < \beta)\)  tends to the limit value
\phantomsection
\label{latex-f4}
\begin{equation}
\label{eq:f4}
 \frac{a + 1}{2^{a}}
 \end{equation}

On the other hand, when \(\lambda_{2}\) is fixed and \(\lambda_{1} \to \infty\), then
the limit of \(\Prob(P_{\lambda}; \alpha < \beta)\) is 1.
\end{theorem}
\begin{proof}
Apply Lemma \ref{lemma-flm-tworows}.
\end{proof}

We will somewhat sloppily say that (\ref{eq:f4}) is the 
the probability that \(\alpha=(1,a)\) is ordered before \(\beta=(2,1)\) in a SYT on an ``infinitely long'' two-row partition \(\lambda\); note that we demand that the two arms have grown towards infinity with bounded difference. In particular, it is the
limit, as \(n \to \infty\), of \(\Prob(P_{(n,n)} ; \alpha < \beta)\). 
We tabulate this limit value for a few values of \(a\) in Table \ref{tab-tworow-b1}.
\begin{table}[htbp]
\caption{\label{tab-tworow-b1}Probability that \(\alpha=(1,a)\) is ordered before \(\beta=(2,1)\) in an infinitely long two-row partition \(\lambda\)}
\centering
\begin{tabular}{lrrrrrrrrr}
a & 2 & 3 & 4 & 5 & 6 & 7 & 8 & 9 & 10\\
Probability & 3/4 & 1/2 & 5/16 & 3/16 & 7/64 & 1/16 & 9/256 & 5/256 & 11/1024\\
\end{tabular}
\end{table}

\begin{mexample}
If we instead put \(\lambda=(\lambda_{1},\lambda_{2}) = t(2,1)\) and let \(t \to \infty\), then
\(\Prob(P_{\lambda}; \alpha < \beta) \to (2^{{a+1} -1})3^{{-a}}\). 
\label{ex-skewlimit}
\end{mexample}

More generally,  we conjecture the following:
\begin{conjecture}
Let \(p \ge q \ge 1\) be two coprime integers. Then
\begin{equation}
  \lim_{{t \to \infty}} \Prob(P_{(tp,tq)} ; (1,a) < (2,1))
  = \frac{p^{a+1} - q^{a+1}}{(p-q)(p+q)^a}
  = \frac{\sum_{{k=0}}^a p^{a-k}q^k}{(p+q)^a}.
\end{equation}
\end{conjecture}
Setting \(p=q=1\) recovers \ref{eq:f4}, as expected.
\subsubsection{The cell in the first row precedes the cell in the second row, general case}
\label{sec:orgf0bd911}
When \(\lambda=(\lambda_{1},\lambda_{2})\), \(\alpha=(1,a)\), \(\beta=(2,b)\), the ``blocking expansion'' contains
several terms; furthermore, it may be convenient to ``reduce'' \(f^{\lambda_{1}/ (a,t)}\)
to \(f^{(\lambda_{1} -t, \lambda_{2} -t) /(a - t)}\), which have the same value, but is a skew
partition whose second part has only one row; we have previously dealt with
those in Lemma \ref{lemma-flm-tworows}.
As an example, when \(\lambda=(10,7)\), \(\alpha=(1,7)\), \(\beta=(2,4)\), the picture is
\begin{displaymath}
    \ydiagram[*(green) \bullet]
    {6+1, 3+1}
    *[*(white)]{10,7}.
\end{displaymath}

\begin{lemma}
Let \(\lambda = (\lambda_{1}, \lambda_{2})\) and \(\mu=(\mu_{1},\mu_{2}) \subset \lambda\) be partitions, with \(\mu_{2} > 0\).
Put \(\tilde{\lambda} = (\lambda_{1} - \mu_{2}, \lambda_{2} - \mu_{2})\), \(\tilde{\mu} = (\mu_{1} - \mu_{2})\).
Then
\begin{equation}
\label{eq-red-tworow}
  \frac{f^{\tilde{\lambda}}}{\fla}
  = \frac{\lvert \tilde{\lambda} \rvert !}{\lvert {\lambda} \rvert !} \frac{H(\lambda)}{H(\tilde{\lambda})}
  = \frac{(\lambda_{1} + \lambda_{2} -2 \mu_{2})!}{(\lambda_{1} + \lambda_{2})!}
  \frac{(\lambda_{1} + 1)! \lambda_{2}!}{(\lambda_{1} + 1 - \mu_{2})! (\lambda_{2} - \mu_{2})!}
\end{equation}
For fixed \(\mu\), as \(\lambda \to (+\infty,+\infty)\) with bounded difference \(\lambda_{1} - \lambda_{2}\),
\begin{equation}
\label{eq-red-tworow-lim}
  \frac{f^{\tilde{\lambda}}}{\fla} \to 2^{-2\mu_{2}}
\end{equation}
\label{lemma-tworow-reduce}
\end{lemma}
\begin{proof}
When reducing we remove \(\mu_{2}\) cells from both rows, hence
\(\frac{\lvert \tilde{\lambda} \rvert !}{\lvert {\lambda} \rvert !}
= \frac{(\lambda_{1} + \lambda_{2} -2 \mu_{2})!}{(\lambda_{1} + \lambda_{2})!}\).

Recall from Lemma \ref{lemma-fl-tworows} that
\(H_{\lambda}(\lambda) = \frac{(\lambda_{1} + 1)!}{1 +\lambda_{1} - \lambda_{2}} \lambda_{2}!\).
Thus
\begin{equation*}
  \begin{split}
    \frac{H(\lambda)}{H(\tilde{\lambda})}
    & = \frac{\frac{(\lambda_{1} + 1)!}{1 + \lambda_{1} - \lambda_{2}} \lambda_{2}!}
            {\frac{(\lambda_{1} + 1 - \mu_{2})!}{1 + \lambda_{1} - \lambda_{2}} (\lambda_{2} - \mu_{2})! }
      \\
    & = \frac{(\lambda_{1} + 1)! \lambda_{2}!}{(\lambda_{1} + 1 - \mu_{2})! (\lambda_{2} - \mu_{2})!}
  \end{split}
\end{equation*}
Combining these two observations yields (\ref{eq-red-tworow}).

The limit (\ref{eq-red-tworow-lim}) was calculated by Maxima \autocite{maxima}.
\end{proof}

\begin{theorem}
Let \(\lambda = (\lambda_{1}, \lambda_{2})\), \(\alpha = (1,a)\), \(\beta = (2,b)\) with \(1 < b < \lambda_{2} < a < \lambda_{1}\).
Let \(P=P_{\lambda} = P_{(\lambda_{1}, \lambda_{2})}\). Then
\begin{equation}
\label{eq:fexp-ag-bg}
  \begin{split}
    \Prob(P; \alpha < \beta)  & = \frac{f^{\lambda / (a)} + 
    \sum_{t=1}^{b-1} f^{(a-1,t)}  f^{(\lambda_{1}-t,\lambda_{2}-t) / (a-t)}}{\fla} \\
                    & =  \frac{f^{\lambda / (a)}}{\fla} +
    \sum_{t=1}^{b-1} f^{(a-1,t)}  \frac{f^{(\lambda_{1}-t,\lambda_{2}-t) / (a-t)}}{f^{\tilde{\lambda}}} \frac{f^{\tilde{\lambda}}}{\fla}
  \end{split}
\end{equation}
and for any positive integer \(r\),
\begin{equation}
\label{eq:fexp-prob}
   \lim_{\substack{(\lambda_{1}, \lambda_{2}) \to (\infty,\infty) \\ \lambda_{2} \le \lambda_{1} \le \lambda_{2} + r}} \Prob(P; \alpha < \beta)  =
2^{-a} \left( (a+1) + \sum_{t=1}^{b-1} \frac{(a-1+t)!}{a!t!}(a-t)(a-t+1) 2^{-t} \right)
\end{equation}
\label{thm-tworows-prob}
\end{theorem}
\begin{proof}
Using Theorem  \ref{thm-tworows} we have that
\begin{equation*}
\label{eq:fexp-ag-bg-pf}
  \begin{split}
    \Prob(P; \alpha < \beta)
    & = \frac{e(P;\alpha < \beta)}{e(P)} \\
    & = \frac{1}{\fla} \sum_{t=0}^{b-1} f^{(a-1,t)} f^{\lambda / (a,t)} \\
    & = \frac{1}{\fla} f^{(a-1)} f^{\lambda / (a)} +
      \frac{1}{\fla} \sum_{t=1}^{b-1} f^{(a-1,t)} f^{(\lambda_{1}-t,\lambda_{2}-t) / (a-t)} \\
    & = \frac{f^{\lambda / (a)}}{\fla} +
      \sum_{t=1}^{b-1} f^{(a-1,t)} \frac{f^{(\lambda_{1}-t,\lambda_{2}-t) / (a-t)}}{\fla} \\
    & = \frac{f^{\lambda / (a)}}{\fla} +
      \sum_{t=1}^{b-1} f^{(a-1,t)}
      \frac{f^{(\lambda_{1}-t,\lambda_{2}-t) / (a-t)}}{f^{(\lambda_{1}-t,\lambda_{2}-t)}}
      \frac{f^{(\lambda_{1}-t,\lambda_{2}-t)}}{\fla}
      \end{split}.
\end{equation*}

From Lemma \ref{lemma-flm-tworows} we get that
\begin{align*}
  \frac{f^{\lambda / (a)}}{\fla}  & \to (a + 1)2^{-a} \\
  \frac{f^{(\lambda_{1}-t,\lambda_{2}-t) / (a-t)}}{f^{(\lambda_{1}-t,\lambda_{2}-t)}}  & \to (a-t+1)2^{-(a-t)}
\end{align*}
as \((\lambda_{1}, \lambda_{2}) \to (\infty,\infty)\) while \(\lambda_{2} \le \lambda_{1} \le \lambda_{2} + r\) for some fixed \(r\).

Lemma \ref{lemma-fl-tworows} gives that
\begin{equation*}
 f^{(a - 1, t)} = \frac{(a - 1 + t)!(a - t)}{a!t!}.
\end{equation*}

Finally, Lemma \ref{lemma-tworow-reduce} shows that
\begin{equation*}
 \frac{f^{(\lambda_{1}-t,\lambda_{2}-t)}}{\fla} \to 2^{-2t}
\end{equation*}
as \((\lambda_{1}, \lambda_{2}) \to (\infty,\infty)\) with \(\lambda_{1} - \lambda_{2}\) bounded.

Hence, combining, we arrive at
\begin{equation*}
\label{eq-tworows-prob-limit-general}
  \begin{split}
    \lim_{\lambda \to (\infty,\infty)}  \Prob(P; \alpha < \beta)
    & = \lim_{\lambda \to (\infty,\infty)} \left(
      \frac{f^{\lambda / (a)}}{\fla}
    + \sum_{t=1}^{b-1} f^{(a-1,t)}
      \frac{f^{(\lambda_{1}-t,\lambda_{2}-t) / (a-t)}}{f^{(\lambda_{1}-t,\lambda_{2}-t)}}
      \frac{f^{(\lambda_{1}-t,\lambda_{2}-t)}}{\fla} \right) \\
    & = \lim_{\lambda \to (\infty,\infty)} \frac{f^{\lambda / (a)}}{\fla} \\
    & +  \sum_{t=1}^{b-1} f^{(a-1,t)}
      \lim_{\lambda \to (\infty,\infty)} \frac{f^{(\lambda_{1}-t,\lambda_{2}-t) / (a-t)}}{f^{(\lambda_{1}-t,\lambda_{2}-t)}}
      \lim_{\lambda \to (\infty,\infty)}\frac{f^{(\lambda_{1}-t,\lambda_{2}-t)}}{\fla}  \\
    & = (a+1)2^{-a}  +
      \sum_{t=1}^{b-1} \frac{(a-1+t)!(a-t)}{a!t!} (a-t+1)2^{-(a-t)} 2^{-2t} \\
    & = 2^{-a} \left( (a+1) + \sum_{t=1}^{b-1} \frac{(a-1+t)!}{a!t!}(a-t)(a-t+1) 2^{-t} \right)
  \end{split}
\end{equation*}
as \((\lambda_{1}, \lambda_{2}) \to (\infty,\infty)\) with \(\lambda_{1} - \lambda_{2}\) bounded.
\end{proof}

\begin{remark}
We can substitute (\ref{eq-flamutw-cg}) in (\ref{eq:fexp-ag-bg}) to obtain an explicit formula
involving binomial coefficients.
\end{remark}
\subsubsection{Examples}
\label{sec:org2941d6e}
The  probabilities
\(\Prob({{P_{\lambda}}}; \alpha < \beta)\)
 with fixed \(\alpha=(1,a)\), \(\beta=(2,b)\),
\(\lambda_{1} \ge \lambda_{2}\) varying, can be arranged as a lower-triangular matrix.
We have already shown the case \(a=2\), \(b=1\).
Here are a few others. For simplicity, \(\lambda_{1},\lambda_{2}\) starts at \(a\).

\(a=3\), \(b=2\)
\begin{displaymath}
 \left(\begin{array}{rrrrrrrrrr}
\frac{3}{5} & 0 & 0 & 0 & 0 & 0 & 0 & 0 & 0 & 0 \\
\frac{5}{7} & \frac{5}{7} & 0 & 0 & 0 & 0 & 0 & 0 & 0 & 0 \\
\frac{11}{14} & \frac{16}{21} & \frac{16}{21} & 0 & 0 & 0 & 0 & 0 & 0 & 0 \\
\frac{5}{6} & \frac{4}{5} & \frac{26}{33} & \frac{26}{33} & 0 & 0 & 0 & 0 & 0 & 0 \\
\frac{13}{15} & \frac{137}{165} & \frac{241}{297} & \frac{115}{143} & \frac{115}{143} & 0 & 0 & 0 & 0 & 0 \\
\frac{49}{55} & \frac{47}{55} & \frac{119}{143} & \frac{821}{1001} & \frac{53}{65} & \frac{53}{65} & 0 & 0 & 0 & 0 \\
\frac{10}{11} & \frac{125}{143} & \frac{851}{1001} & \frac{76}{91} & \frac{43}{52} & \frac{14}{17} & \frac{14}{17} & 0 & 0 & 0 \\
\frac{12}{13} & \frac{81}{91} & \frac{709}{819} & \frac{309}{364} & \frac{57}{68} & \frac{382}{459} & \frac{268}{323} & \frac{268}{323} & 0 & 0 \\
\frac{85}{91} & \frac{411}{455} & \frac{80}{91} & \frac{205}{238} & \frac{433}{510} & \frac{815}{969} & \frac{1351}{1615} & \frac{111}{133} & \frac{111}{133} & 0 \\
\frac{33}{35} & \frac{32}{35} & \frac{106}{119} & \frac{89}{102} & \frac{555}{646} & \frac{4117}{4845} & \frac{561}{665} & \frac{1229}{1463} & \frac{135}{161} & \frac{135}{161}
\end{array}\right)
 \end{displaymath}

\(a=4\), \(b=2\)
\begin{displaymath}
 \left(\begin{array}{rrrrrrrrrr}
\frac{2}{7} & 0 & 0 & 0 & 0 & 0 & 0 & 0 & 0 & 0 \\
\frac{17}{42} & \frac{17}{42} & 0 & 0 & 0 & 0 & 0 & 0 & 0 & 0 \\
\frac{1}{2} & \frac{31}{66} & \frac{31}{66} & 0 & 0 & 0 & 0 & 0 & 0 & 0 \\
\frac{19}{33} & \frac{157}{297} & \frac{73}{143} & \frac{73}{143} & 0 & 0 & 0 & 0 & 0 & 0 \\
\frac{7}{11} & \frac{83}{143} & \frac{551}{1001} & \frac{7}{13} & \frac{7}{13} & 0 & 0 & 0 & 0 & 0 \\
\frac{98}{143} & \frac{626}{1001} & \frac{107}{182} & \frac{59}{104} & \frac{19}{34} & \frac{19}{34} & 0 & 0 & 0 & 0 \\
\frac{66}{91} & \frac{544}{819} & \frac{453}{728} & \frac{81}{136} & \frac{533}{918} & \frac{371}{646} & \frac{371}{646} & 0 & 0 & 0 \\
\frac{69}{91} & \frac{127}{182} & \frac{311}{476} & \frac{127}{204} & \frac{584}{969} & \frac{191}{323} & \frac{78}{133} & \frac{78}{133} & 0 & 0 \\
\frac{11}{14} & \frac{173}{238} & \frac{139}{204} & \frac{837}{1292} & \frac{605}{969} & \frac{81}{133} & \frac{878}{1463} & \frac{96}{161} & \frac{96}{161} & 0 \\
\frac{55}{68} & \frac{1381}{1836} & \frac{5477}{7752} & \frac{1301}{1938} & \frac{331}{513} & \frac{131}{209} & \frac{311}{506} & \frac{503}{828} & \frac{139}{230} & \frac{139}{230}
\end{array}\right)
 \end{displaymath}

\(a=5\), \(b=2\)
\begin{displaymath}
 \left(\begin{array}{rrrrrrrrrr}
\frac{5}{42} & 0 & 0 & 0 & 0 & 0 & 0 & 0 & 0 & 0 \\
\frac{13}{66} & \frac{13}{66} & 0 & 0 & 0 & 0 & 0 & 0 & 0 & 0 \\
\frac{3}{11} & \frac{107}{429} & \frac{107}{429} & 0 & 0 & 0 & 0 & 0 & 0 & 0 \\
\frac{49}{143} & \frac{303}{1001} & \frac{41}{143} & \frac{41}{143} & 0 & 0 & 0 & 0 & 0 & 0 \\
\frac{58}{143} & \frac{709}{2002} & \frac{373}{1144} & \frac{139}{442} & \frac{139}{442} & 0 & 0 & 0 & 0 & 0 \\
\frac{6}{13} & \frac{293}{728} & \frac{19}{52} & \frac{457}{1326} & \frac{217}{646} & \frac{217}{646} & 0 & 0 & 0 & 0 \\
\frac{93}{182} & \frac{2767}{6188} & \frac{5351}{13260} & \frac{364}{969} & \frac{581}{1615} & \frac{6}{17} & \frac{6}{17} & 0 & 0 & 0 \\
\frac{66}{119} & \frac{41}{84} & \frac{1705}{3876} & \frac{9847}{24225} & \frac{870}{2261} & \frac{9258}{24871} & \frac{1122}{3059} & \frac{1122}{3059} & 0 & 0 \\
\frac{121}{204} & \frac{4073}{7752} & \frac{2297}{4845} & \frac{141}{323} & \frac{7287}{17765} & \frac{3777}{9614} & \frac{669}{1748} & \frac{87}{230} & \frac{87}{230} & 0 \\
\frac{143}{228} & \frac{542}{969} & \frac{3923}{7752} & \frac{451}{969} & \frac{91}{209} & \frac{7963}{19228} & \frac{2}{5} & \frac{1171}{2990} & \frac{803}{2070} & \frac{803}{2070}
\end{array}\right)
 \end{displaymath}

As before, we can look at a matrix plot of these cases.
They are shown in Figure \ref{fig:Mat-tworows-c1-export}.

\begin{center}
\begin{figure}[htbp]
\centering
\includegraphics[width=1.1\textwidth]{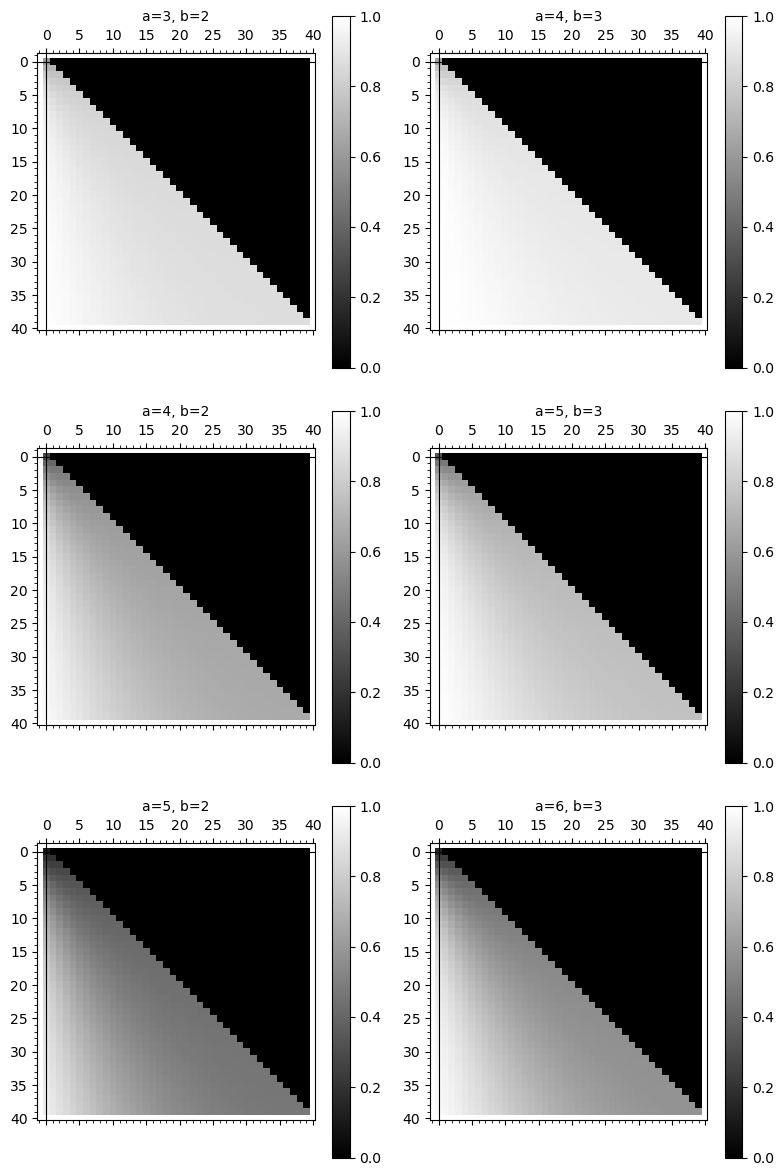}
\caption{\label{fig:Mat-tworows-c1-export}Probabilities for two-row partitions,  \(2 \le b \le 3\), \(b+1 \le a \le b+3\)}
\end{figure}
\end{center}
\subsubsection{Limit probabilities}
\label{sec:orgb090e0d}
We tabulate, for small values of \(a\) and \(b\), the limit probabilities predicted by Theorem \ref{thm-tworows-prob}, namely 
\[
\lim_{(\lambda_1,\lambda_2) \to (\infty,\infty)}\Prob(P_{(\lambda_{1},\lambda_{2})}; (1,a) < (2,b)),
\]
where \(\lambda_{1}, \lambda_{2}\) tends to infinity with bounded non-negative difference \(\lambda_{1} - \lambda_{2}\). 
\begin{table}[htb]
\begin{center}
\label{tab:limprobsmall}
\caption{Probabilites for $(1,a) < (2,b)$ in infinite Catalan poset}
\begin{tabular}{|r|r|r|r|r|r|r|r|r|r|} \hline
 & b=1 & b=2 & b=3 & b=4 & b=5 & b=6 & b=7 & b=8 & b=9 \\ \hline \hline
a=2 & $\frac{3}{4}$ & $0$ & $0$ & $0$ & $0$ & $0$ & $0$ & $0$ & $0$ \\ \hline
a=3 & $\frac{1}{2}$ & $\frac{7}{8}$ & $0$ & $0$ & $0$ & $0$ & $0$ & $0$ & $0$ \\ \hline
a=4 & $\frac{5}{16}$ & $\frac{11}{16}$ & $\frac{59}{64}$ & $0$ & $0$ & $0$ & $0$ & $0$ & $0$ \\ \hline
a=5 & $\frac{3}{16}$ & $\frac{1}{2}$ & $\frac{25}{32}$ & $\frac{121}{128}$ & $0$ & $0$ & $0$ & $0$ & $0$ \\ \hline
a=6 & $\frac{7}{64}$ & $\frac{11}{32}$ & $\frac{79}{128}$ & $\frac{107}{128}$ & $\frac{491}{512}$ & $0$ & $0$ & $0$ & $0$ \\ \hline
a=7 & $\frac{1}{16}$ & $\frac{29}{128}$ & $\frac{59}{128}$ & $\frac{89}{128}$ & $\frac{223}{256}$ & $\frac{991}{1024}$ & $0$ & $0$ & $0$ \\ \hline
a=8 & $\frac{9}{256}$ & $\frac{37}{256}$ & $\frac{337}{1024}$ & $\frac{281}{512}$ & $\frac{3073}{4096}$ & $\frac{3667}{4096}$ & $\frac{15955}{16384}$ & $0$ & $0$ \\ \hline
a=9 & $\frac{5}{256}$ & $\frac{23}{256}$ & $\frac{29}{128}$ & $\frac{849}{2048}$ & $\frac{2523}{4096}$ & $\frac{1619}{2048}$ & $\frac{7477}{8192}$ & $\frac{32053}{32768}$ & $0$ \\ \hline
a=10 & $\frac{11}{1024}$ & $\frac{7}{128}$ & $\frac{155}{1024}$ & $\frac{309}{1024}$ & $\frac{7947}{16384}$ & $\frac{5475}{8192}$ & $\frac{26905}{32768}$ & $\frac{30337}{32768}$ & $\frac{128641}{131072}$ \\ \hline
\end{tabular}
\end{center}
\end{table}

Pictorially, the limit probabilities are shown in Figure \ref{fig-Mat-limprob-export}.

\begin{figure}[htbp]
\centering
\includegraphics[width=0.7\textwidth]{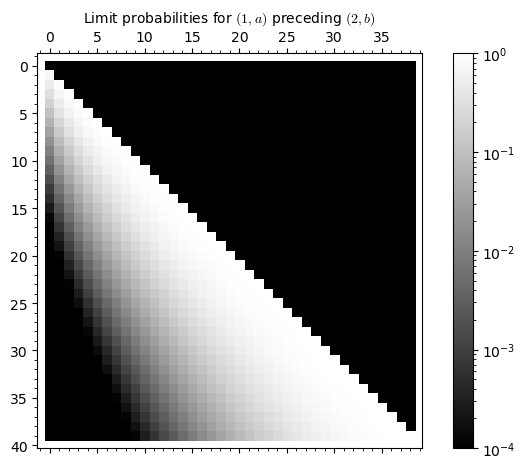}
\caption{\label{fig-Mat-limprob-export}Probability that \((1,a)\) precedes \((2,b)\) in an infinite two-row SYT}
\end{figure}

\begin{remark}
Similar to our discussion about the \(\alpha=(1,a)\), \(\beta=(2,1)\) case,
one could try to determine the limits of
\(\Prob_{\lambda}(\alpha < \beta)\) as \(\lambda = (tp,tq)\) with \(p \ge q \ge 1\) coprime integers,
\(t \to \infty\).
\end{remark}
\section{Concluding remarks}
\label{sec:orgef37190}
Partition posets constitute a class of posets for which the problem of calculating
exact poset probabilities is tractable. We have done so for two-row partitions, and the methods presented here should be suitable for partitions with three or more rows.
The number of cases involving the placement of \(\alpha,\beta\) in different rows will increase,
as will the length of the associated blocking expansions, but this seems to be
very surmountable obstacles.

One could also extend the blocking-partitions method to deal with the probability
\[
\Prob(P_{\lambda}; \alpha_1 < \alpha_2 \cdots < \alpha_k),
\]
where \(\set{\alpha_{1}, \alpha_2, \dots, \alpha_k} \subset P_\lambda\)
is a \(k\)-element antichain. In particular, the case where \(k\) is also
the number of rows of \(\lambda\) would be interesting.

We have calculated limit probabilities
\[
\lim_{(\lambda_{1},\lambda_2) \to (\infty, \infty)}\Prob(P_{(\lambda_{1},\lambda_2)}; \alpha < \beta)\]
for fixed \(\alpha, \beta\) and \(\lambda = \lambda(t) = (t+r, r)\), \(t \to \infty\), \(r\) fixed.
As the examples \ref{example-a2} and \ref{ex-skewlimit} show, letting \(\lambda\) tend to infinity
in other ways gives a different limit. We have chosen the limit with bounded difference
because it is the most natural, and the easiest to compute.
However, for the suggested generalisation
\[
\Prob(P_{(\lambda_1, \dots, \lambda_k)}; \alpha_1 < \alpha_2 \cdots < \alpha_k),
\]
other limits may be more natural and tractable.

\printbibliography
\end{document}